\documentclass[11pt,reqno]{amsart}

\usepackage{amsrefs}
\usepackage{amsmath}
\usepackage{amssymb}
\usepackage{hyperref}
\usepackage[utf8]{inputenc}
\usepackage{stackrel}
\usepackage{color}
\usepackage{xcolor}
\usepackage{graphics,graphicx}
\usepackage{comment}

\newtheorem{theorem}{Theorem}[section]
\newtheorem{lemma}[theorem]{Lemma}

\newtheorem{definition}[theorem]{Definition}

\newtheorem{proposition}[theorem]{Proposition}
\newtheorem{remark}[theorem]{Remark}
\newtheorem{corollary}[theorem]{Corollary}

\DeclareMathOperator{\spann}{span}
\DeclareMathOperator{\id}{id}

\DeclareMathOperator{\ind}{ind}
\DeclareMathOperator{\co}{co}

\begin{document}
 
\title{Geodesic fields for Pontryagin type $C^0$-Finsler manifolds}

\author{Ryuichi Fukuoka}
\address{Department of Mathematics, State University of Maring\'a,
87020-900, Maring\'a, PR, Brazil \\ email: rfukuoka@uem.br}

\author{Hugo Murilo Rodrigues}
\address{Department of Mathematics, State University of Maring\'a,
87020-900, Maring\'a, PR, Brazil \\ email: hugo\_murilo@hotmail.com}

\date{\today}

\begin{abstract}
Let $M$ be a differentiable manifold, $T_xM$ be its tangent space at $x\in M$ and $TM=\{(x,y);x\in M;y \in T_xM\}$ be its tangent bundle. 
A $C^0$-Finsler structure is a continuous function $F:TM \rightarrow \mathbb [0,\infty)$ such that $F(x,\cdot): T_xM \rightarrow [0,\infty)$ is an asymmetric norm. 
In this work we introduce the Pontryagin type $C^0$-Finsler structures, which are structures that satisfy the minimum requirements of Pontryagin's maximum principle for the problem of minimizing paths.
We define the extended geodesic field $\mathcal E$ on the slit cotangent bundle $T^\ast M\backslash 0$ of $(M,F)$, which is a generalization of the geodesic spray of Finsler geometry.
We study the case where $\mathcal E$ is a locally Lipschitz vector field.
We show some examples where the geodesics are more naturally represented by $\mathcal E$ than by a similar structure on $TM$.
Finally we show that the maximum of independent Finsler structures is a Pontryagin type $C^0$-Finsler structure where $\mathcal E$ is a locally Lipschitz vector field.
\end{abstract}

\date{}



\maketitle

\section{Introduction}

Let $M$ be a differentiable manifold, $T_xM$ be its tangent space at $x\in M$ and $T^\ast_xM$ be its cotangent space at $x$.
Denote the tangent and cotangent bundle of $M$ by $TM=\{(x,y);x\in M, y\in T_xM \}$ and $T^\ast M = \{ (x,\alpha); x\in M, \alpha \in T^\ast_x M\}$ respectively. 
The slit tangent bundle and the slit cotangent bundle of $M$ will be denoted by $TM\backslash 0$ and $T^\ast M\backslash 0$ respectively.
A $C^0$-Finsler structure on $M$ is a continuous function $F:TM \rightarrow \mathbb R$ such that $F(x,\cdot)$ is an asymmetric norm (a norm without the symmetry condition $F(x,y)=F(x,-y)$). 
$C^0$-Finsler structures arise naturally in the study of intrinsic metrics on homogeneous spaces. 
In \cite{Berestovskii1, Berestovskii2}, the author proves that if $M$ is a locally compact and locally contractible homogeneous space endowed with an invariant intrinsic metric $d_M$, then $(M,d_M)$ is isometric to a left coset manifold $G/H$ of a Lie group $G$ by a compact subgroup $H < G$ endowed with a $G$-invariant $C^0$-Carnot-Carath\'{e}odory-Finsler metric.
He also proved that if every orbit of one-parameter subgroups of $G$ (under the natural action $a: G \times G/H \rightarrow G/H$) is rectifiable, then $d_M$ is $C^0$-Finsler.

Finsler geometry and Riemannian geometry have several points in common. 
Differential calculus is one of their main tools and, as a consequence, geometrical objects like geodesics are very similar.
For instance, given $x\in M$ and $y\in T_xM$, there exists a unique geodesic $\gamma:(-\epsilon, \epsilon) \rightarrow M$ such that $\gamma(0) = x$ and $\gamma^\prime(0)=y$.
The study of connections and curvatures in Finsler geometry resembles the study of those objects in Riemannian geometry.
On the other hand there are some differences between them as well, like the non-existence of a canonical volume element in Finsler geometry, what makes geometric analysis as developed in Riemannian geometry less natural in Finsler geometry.

Geometrical properties of $C^0$-Finsler manifolds can be very different from those of Finsler manifolds.
For instance $\mathbb R^2$ endowed with the maximum norm can be canonically identified with a $C^0$-Finsler manifold $(M,F)$.
Given $x\in \mathbb R^2$ and a vector $y$ in the tangent space of $x$, there exist infinitely many geodesics $\gamma:(-\epsilon, \epsilon) \rightarrow \mathbb R^2$ with constant speed such that $\gamma (0) = x$ and $\gamma^\prime(0)=y$.
On the other hand, we have a large family of projectively equivalent $C^0$-Finsler manifolds on $\mathbb R^2$ such that for every pair of points, there exists a unique minimizing path connecting them:
All minimizing paths are line segments parallel to the vectors $(0,1)$, $(\sqrt{3}/2,1/2)$ or $(-\sqrt{3}/2, 1/2)$ or else a concatenation of two of these line segments (see \cite{Fukuoka-large-family}). 
Therefore if $y$ isn't parallel to one of these three vectors, then there isn't any geodesic satisfying $\gamma (0) = x$ and $\gamma^\prime (0) = y$.
If $y$ is parallel to one of these vectors, then there exist infinitely many minimizing paths satisfying $\gamma (0) = x$ and $\gamma^\prime(0) = y$.

Control theory, Hamiltonian formalism, Legendre transformation and Pontryagin's maximum principle (PMP) have been valuable tools in order to study $C^0$-sub-Finsler manifolds $(M,F)$ and also in order to develop Riemannian and Finsler geometry under a new perspective (See, for instance, \cite{Agrachev-Barilari-Rizzi}, \cite{Agrachev-Gamkrelidze-feedback-1}, \cite{Lee-displacement}, \cite{MatveevTroyanov} and \cite{Ohta-Hamiltonian} among many other works).
Control theory and PMP are specially very well suited to overcome the lack of vertical smoothness of $F$ as well as to deal with the restriction of admissible curves.
In general these problems are formalized through the Hamiltonian formalism on $T^\ast M$.
It is important to notice that in order to apply the PMP, $(M,F)$ must satisfy some type of ``horizontal smoothness''.
For instance, this feature is assured for left invariant $C^0$-sub-Finsler structures on Lie groups and for sub-Riemannian structures. 
In what follows, we present a (certainly incomplete) list of works related to this paper.

Several geometrical objects of Riemannian manifolds such as geodesics, Jacobi fields, conjugate points, volume, curvatures, geometric inequalities, comparison theorems, etc., have been studied in sub-Riemannian geometry using control theory and PMP (see, for instance, \cite{Agrachev-Barilari-Rizzi}, \cite{Agrachev-Barilari-Boscain}, \cite{Agrachev-Barilari-Paoli}, \cite{Barilari-Chitour}, \cite{Barilari-Rizzi-conjugate},  \cite{Barilari-Rizzi-Inventiones}, \cite{Boscain-Rossi} and references therein).
For a more in depth study of sub-Riemannian geometry, see \cite{Agrachev-Barilari-Boscain}.

Pontryagin extremals and geodesics have been studied and even calculated explicitly in Lie groups endowed with left invariant $C^0$-sub-Finsler structures.
The first case where all the minimizing paths were calculated explicitly was on the two dimensional non-abelian simply connected Lie group (see \cite{Gribanova}). 
Several cases were studied since then (see, for instance, \cite{Ardentov-LeDonne-Sachkov}, \cite{Ardentov-Lokutsievskiy-Sachkov}, \cite{Berestovskii-Zubareva-Engel}, \cite{Hakavuori-Infinite-geodesics}, \cite{Lokutsievskiy-convex}, \cite{Sachkov-bang-bang-Cartan}, \cite{Sachkov-bang-bang-Engel}) and they emphasize that several phenomena that do not occur in the Finsler case can happen in this setting, as the sudden change of derivatives along minimizing paths.
Here, as it frequently happens in the study of Lie groups endowed with invariant geometrical structures, the problem is usually faithfully represented on its Lie algebra and on its dual space.

In \cite{Agrachev-Gamkrelidze-feedback-1} the authors proposed a systematic way to study a wide range of geometrical objects like $C^0$-sub-Finsler manifolds, pseudo-Riemannian manifolds, etc. 
In the particular case of $C^0$-sub-Finsler manifolds, the control region is given by a convex and compact subset of some Euclidean space with the origin in its interior.
The closed unit balls on the distribution are given in terms of a ``smooth deformation'' of the control region. 
This is, according to our best knowledge, the first time where this type of general approach was proposed. 

In \cite{Zelenko-Li} the authors considered smooth submanifolds of $TM$ transversal to the fibers as the geometrical structure, which includes smooth (not necessarily strongly convex) sub-Finsler structures.

In \cite{Barilari-Boscain-LeDonne-Sigalotti}, the authors study a particular case of $C^0$-sub-Finsler structure, which is constructed from a product $M\times \mathbb (\mathbb R^k,\Vert \cdot\Vert)$ and a smooth morphism of vector bundles $f: M \times \mathbb R^k \rightarrow TM$. 
The map $f$ induces a distribution $\mathcal D = f(M \times U)$ on $TM$, and $(M, \mathcal D)$  is endowed with a $C^0$-sub-Finsler metric $\Vert v\Vert_{sF} = \inf \{\Vert w\Vert; f(p,w) = v\}$.
The problem of minimizing path can be written in terms of a control system with a finite number of controls.

Let $C \subset \mathbb R^k$ be a closed control region and $f: M \times U \rightarrow TM$ be a Lipschitz function such that $f_u = f(\cdot, u)$ are smooth vector fields for every $u \in C$.
Suppose that $(x,u) \mapsto \partial f/\partial x (x,u)$ is continuous. 
In \cite{Agrachev-Lee}, the authors study the Bolza problem using this setting, which includes the case of Lipschitz $C^0$-Finsler structures.
This case is the most similar to the theory we develop here.

From the aforementioned works, it is clear that the use of PMP for this kind of problems is usual nowadays.
However it is also clear that there is not a standard way to study $C^0$-sub-Finsler structures with ``horizontal smoothness'' 
outside the sub-Riemannian geometry and the left invariant $C^0$-sub-Finsler structures on Lie groups.
This is a natural situation because the geometrical structures that can be studied with PMP can vary and even different problems on the same manifold needs its own formulation.

In this work we study Pontryagin type $C^0$-Finsler structures on differentiable manifolds.
We suppose that the minimum requirements of PMP are in place: 
$F$ is continuous and there are a family of unit vector fields such that the PMP holds. 
We apply the PMP for the time-optimal problem, and as the result, we obtain the extended geodesic field $\mathcal E$ on $T^\ast M\backslash 0$, which is a generalization of the geodesic spray of Finsler geometry.
The integral curves $\gamma(t)$ of the differential inclusion $\gamma^\prime(t) = \mathcal E(\gamma(t))$ are the Pontryagin extremals and their projection on $M$ are the candidates to be minimizing paths parametrized by arclength.

Now we present the contributions of this work for the theory of Pontryagin type $C^0$-Finsler manifolds.

\begin{enumerate}
\item Let $F_*:T^\ast M \rightarrow \mathbb R$ be the fiberwise dual $C^0$-Finsler structure of $F$ on $M$.
In Theorem \ref{E generaliza G} we prove that $F_\ast \mathcal E$ is a full generalization of the geodesic spray of Finsler manifolds.
This feature makes $F_\ast \mathcal E$ a natural candidate to bring elements of Finsler geometry to $C^0$-Finsler manifolds;
\item In Theorem \ref{dF* localmente Lipschitz}, we determine conditions on $(M,F)$ such that $\mathcal E$ (and $F_\ast \mathcal E$) is a locally Lipschitz vector field.
Convex analysis, the dual norm $F_\ast$ and the inverse of the Lengendre transform are used in order to assure the regularity of the vertical part of $\mathcal E$ (See Section \ref{asymmetric norms}).
\item In Subsection \ref{A non-strictly convex case} we show that $\mathcal E$ isn't just a theoretical object. 
It is possible to calculate geodesics explicitly even when $\mathcal E$ isn't a vector field.
\item In Subsection \ref{A strongly convex example} we present an example where every minimizing path parametrized by arclength can be represented as a projection of a integral curve of a locally Lipschitz vector field $\mathcal E$ on $T^\ast M\backslash 0$ but they can not be represented as a projection of a integral curve of a locally Lipschitz vector field on $TM$;
\item We present examples where the theory is applicable: Horizontally $C^1$ family of asymmetric norms (Section \ref{Horizontally C1 family of asymmetric norms subsection}), homogeneous spaces (Section \ref{Extended geodesic field on homogeneous spaces}) and the maximum of independent Finsler structures (Subsection \ref{maximum-Finsler}).
In each case, a family of vector fields of the control system is presented explicitly.

\end{enumerate}
It is important to observe that minimizing paths of Items (3) and (4) were calculated originally in \cite{Gribanova}, but we use only $\mathcal E$ in this work.

A relevant question is what happens if we drop the horizontal smoothness of a $C^0$-Finsler structure. 
In this case the theory is much less developed due to the lack of a model theory (like Riemannian geometry) and also due to the lack of tools in order to study variational problems in detail.
In \cite{Fukuoka-Setti, Setti} the authors consider a general $C^0$-Finsler structure $F$ and create a one-parameter family of Finsler structures $(F_\varepsilon)_{\varepsilon \in (0,1)}$ that converges uniformly to $F$ on compact subsets of $TM$.
This smoothing works properly on Finsler structures $F$, that is, several connections of Finsler geometry and the flag curvatures of $(M,F_\varepsilon)$ converges uniformly on compact subsets to the respective objects of $(M,F)$. 

The variation of the velocity of geodesics in the example $(M_1,F_1)$ of Subsection \ref{A non-strictly convex case} is discrete and the corresponding variation in a Finsler manifold $(M_2,F_2)$ is smooth. 
If we consider the $C^0$-Finsler manifold $(M_1\times M_2,F_1 + F_2)$, where $(F_1+ F_2)((x_1,x_2),(y_1,y_2))=F_1(x_1,y_1) + F_2(x_2,y_2)$, then discrete and continuous dynamics takes place in the same space naturally.
Eventually this type of feature of $C^0$-Finsler manifolds can be interesting in order to represent some dynamical systems.

In differential geometry, it is usual to use the term differentiable and smooth for something of class $C^\infty$. 
In this work, the term smooth stands for $C^\infty$ and the term differentiable has the usual meaning, that is, its variation can be locally approximated by a linear map in coordinate systems. 
We make this distinction because we deal with several non-smooth maps.
The exception applies to the terms like differentiable manifold, differentiable structure, etc, because they are widely used.
Differentiable manifolds will be always of class $C^\infty$.

This work is organized as follows: 
In Section \ref{preliminaries} we summarize the theory necessary for the development of this work.
In Section \ref{define pontryagin type} we define the Pontryagin type $C^0$-Finsler manifolds.
In Section \ref{Hamiltonian formalism} we present the extended geodesic field $\mathcal E$ and we prove that minimizing paths on $(M,F)$ parameterized by arclength are necessarily projection of integral curves of $\mathcal E$ (see Theorem \ref{minimizing}).
We also prove that $F_\ast \mathcal E$ is a generalization of the geodesic spray of Finsler geometry (see Theorem \ref{E generaliza G}).
In the next four sections we study conditions that guarantee that $\mathcal E$ is a locally Lipschitz vector field on $T^\ast M \backslash 0$.
Section \ref{asymmetric norms} deals with the geometry of asymmetric norms on vector spaces and its dual asymmetric norm.
In Section \ref{Horizontally C1 family of asymmetric norms subsection} we study the horizontally $C^1$ families of asymmetric norms in order to provide a large family of Pontryagin type $C^0$-Finsler manifolds in Section \ref{A Locally Lipschitz Case}.
In Section \ref{Extended geodesic field on homogeneous spaces} we prove that invariant $C^0$-Finsler structures on homogeneous spaces are of Pontryagin type. 
We also prove that if $F$ restricted to any tangent space is strongly convex, then $\mathcal E$ is a locally Lipschitz vector field (see Theorem \ref{localmente lipschitz homogeneo}).
Section \ref{Exemplos} is devoted to the study of three examples. 
The first two examples are quasi-hyperbolic planes, which were studied extensively in \cite{Gribanova}. 
Our study is focused on the usefulness of $\mathcal E$ instead of doing explicit calculations using PMP as it was done in \cite{Gribanova}.
The last example shows that the maximum of independent Finsler structures are Pontryagin type $C^0$-Finsler structures such that $\mathcal E$ is a locally Lipschitz vector field.
Finally in Section \ref{conclusoes} we leave some open problems for future works.

This work was mostly done during the Ph.D. of the first author under the supervision of the second author at State University of Maring\'a, Brazil.
The authors would like to thank J\'essica Buzatto Prudencio and professors Adriano João da Silva, Fernando Manfio, Josiney Alves de Souza, Luiz Antonio Barrera San Martin and Patr\'\i cia Hernandes Baptistelli for their suggestions. The authors would also like to thank the referee for his/her remarks and suggestions, which helped us to improve the original version of this work.

\section{Preliminaries}
\label{preliminaries}

In this section we fix notations and we present a summary of the results we use in this work.
We make this section short because the prerequisites can be found in the literature. 
For control theory and Pontryagin's maximum principle, see \S 11 and \S 12 of \cite{Pontryagin} and \cite{Sachkov}.
For the Hamiltonian formalism developed on differentiable manifolds, see \cite{Sachkov}.
The Finsler geometry prerequisites can be found in \cite{BaoChernShen} and the convex analysis prerequisites can be found in  \cite{RockafellarTyrrell}.

{\em In this work the Einstein convention for the summation of indices is in place, except in Section \ref{Hamiltonian formalism}, until Theorem \ref{minimizing}, where there are two possible variation of indices.}

\begin{definition}
An asymmetric norm $F$ on a finite dimensional real vector space $V$ is a function $F:V \rightarrow [0,\infty)$ satisfying
\begin{itemize}
\item $F(y) = 0$ iff $y=0$;
\item $F(\lambda y) = \lambda F(y)$ for every $\lambda >0$ and $y \in V$;
\item $F(y_1 +y_2) \leq F(y_1) + F(y_2)$.
\end{itemize}
\end{definition}   
Compare with \cite{Cobzas}.

\begin{definition}
If $F$ is an asymmetric norm on $V$, then we define the following subsets:
\begin{enumerate}
\item $B_F(y,R)=\{z\in\mathbb{R}^n; F(z-y)<R\}$, open ball with center $y$ and radius $R$;
\item $B_F[y,R]=\{z\in\mathbb{R}^n; F(z-y)\leq R\}$, closed ball with center $y$ and radius $R$;
\item $S_F[y,R]=\{z\in\mathbb{R}^n; F(z-y)=R\}$, sphere with center $y$ and radius $R$.
\end{enumerate}
\end{definition}

\begin{definition}
\label{define-c0-Finsler}
A $C^0$-Finsler structure on a differentiable manifold $M$ is a continuous function $F: TM \rightarrow [0,\infty)$ such that $F(x,\cdot): T_xM \rightarrow \mathbb R$ is an asymmetric norm for every $x \in M$.
\end{definition}

\begin{remark}
There are three definitions of Finsler structure in the literature. 
The smooth version (see \cite{BaoChernShen}) is by far the most usual and there are two versions where $F$ is continuous: in one of them $F(x,\cdot)$ is a norm (see \cite{Burago}) and in the other one $F(x,\cdot)$ is an asymmetric norm (see \cite{MatveevTroyanov}).
In this work we use the first definition.
In \cite{CordovaFukuokaNeves, Fukuoka-large-family}, the first author of this work and his collaborators used the term $C^0$-Finsler structure for the second definition of Finsler structure.
The term $C^0$-Finsler structure given in Definition \ref{define-c0-Finsler} coincides with the third definition of Finsler structure above. It was first used in \cite{Fukuoka-Setti} and it is a natural generalization of (smooth) Finsler structure. 
\end{remark}

We denote the open ball in $T_xM$ centered at $y$ with radius $R$ by $B_F(x,y,R)$.
Similar notations hold for closed balls and spheres on tangent spaces.

Now we present the Euler's theorem, that can  be found in \cite{BaoChernShen}. 
It is used several times in this work and it is put here for the sake of convenience.

\begin{theorem}[Euler's theorem]
\label{Euler theorem}
Suppose that $f:\mathbb R^n \rightarrow \mathbb R$ is differentiable in $\mathbb R^n\backslash 0$.
Then the following statements are equivalent:
\begin{enumerate}
\item $f$ is positively homogeneous of degree $r$, that is, $f(\lambda y) = \lambda^r f(y)$ for every $\lambda >0$;
\item The radial directional derivative of $f$ is given by $y^i \frac{\partial f}{\partial y^i} = rf(y)$.
\end{enumerate}
\end{theorem}

Let $(M,F)$ be an $n$-dimensional Finsler manifold.
Let $\phi=(x^1, \ldots, x^n)$ be a coordinate system on $M$ and $d\phi = (x^1, \ldots, x^n, y^1, \ldots, y^n)$ be the natural coordinate system on $TM$ with respect to $\phi$. 
Then $g_{ij} = \frac{1}{2}\frac{\partial^2 F^2}{\partial y^i \partial y^j} : TM\backslash 0 \rightarrow \mathbb R$ are the coefficients of the fundamental tensor of $F$ and
\[
C_{ijk}=\frac{1}{2}\frac{\partial g_{ij}}{\partial y^k}
\]
are the coefficients of the Cartan tensor.
We have that
\begin{equation}
\label{cartan e y zera}
\frac{\partial g_{ij(x,y)}}{\partial y^k} y^i 
= \frac{\partial g_{ij(x,y)}}{\partial y^k} y^j 
= \frac{\partial g_{ij(x,y)}}{\partial y^k} y^k 
= 0
\end{equation}
where $g_{ij(x,y)}$ are the coefficients of the fundamental tensor at $(x,y)$ and $y^i$ are the coordinates of $y$.
The formal Christoffel symbols of second kind of $(M,F)$ are given by
\begin{equation}
\label{Christoffel}
\gamma^{i}_{\ jk} = \frac{1}{2} g^{is}\left( \frac{\partial g_{sj}}{\partial x^k} + \frac{\partial g_{sk}}{\partial x^j} - \frac{\partial g_{jk}}{\partial x^s}\right),
\end{equation}
where $g^{is}$ are the coefficients of the inverse tensor of $g_{is}$.

\begin{remark}[Local existence of geodesics]
In \cite{Mennucci-asymmetric-distances} Mennucci introduced the $r$-intrinsic, $g$-intrinsic and $s$-intrinsic asymmetric metric spaces.
These concepts coincides for metric spaces and also for asymmetric metric spaces that are locally bilipschitz to a metric space. 
He proved that $C^0$-Finsler manifolds are $r$-intrinsic asymmetric metric spaces (in fact, this is only a particular case of his result, because his definition of ``Finsler manifold'' is more general than the definition used in this work).
Observe that $C^0$-Finsler structures are locally bilipschitz to Riemannian metrics and therefore they are also $g$-intrinsic and $s$-intrinsic.

Let $(Z,d_Z)$ be a $r$-intrinsic asymmetric metric space endowed with the topology $\mathcal T$, which is generated by the metric 
\[
d_{\max}(z_1,z_2)=\max\{d_Z(z_1,z_2),d(z_2,z_1)\}.
\]
Let $x,y \in Z$ and $\rho >0$ such that  $d_Z(x,y) \leq \rho$. 
Suppose that the forward balls 
\[
B^+[x,\rho^\prime]=\{z \in Z;d_Z(x,z) \leq \rho^\prime\}
\]
are contained in $\mathcal T$-compact subsets for every $\rho^\prime \leq \rho$. 
In \cite{Mennucci-geodesics} the author proves that there exist a minimizing path connecting $x$ and $y$.
In particular this result holds for $C^0$-Finsler manifolds in general. 

As in the case of intrinsic metric spaces (see for instance \cite{Burago}), Mennucci's proof is existential in nature. 
Control theory and the PMP allow us to obtain more geometrical details in $C^0$-sub-Finsler structures with ``horizontal smoothness'' as well as to calculate geodesics explicitly when $(M,F)$ has enough symmetries. 

\end{remark}

\section{Pontryagin type $C^0$-Finsler manifolds}
\label{define pontryagin type}

In this section we introduce the Pontryagin type $C^0$-Finsler manifolds.
We define them in order to satisfy the minimum requirements of Pontryagin's maximum principle (PMP).
The control region $C$ is the Euclidean unit sphere $S^{n-1}$.

\begin{definition}
\label{horizontally smooth}
A $C^0$-Finsler manifold is of Pontryagin type at $p \in M$ if there exist a neighborhood $U$ of $p$, a coordinate system $\phi = (x^1, \ldots, x^n): U \rightarrow \mathbb R^n$ (with the respective natural coordinate system $\phi_{TU} := d\phi: (x^1, \ldots, x^n,$ $ y^1, \ldots, y^n): TU \rightarrow \mathbb R^{2n}$ of the tangent bundle) and a family of $C^1$ unit vector fields 
\[
\{x \mapsto X_u(x) = (y^1(x^1, \ldots, x^n,u), \ldots, y^n(x^1, \ldots, x^n,u)); u \in S^{n-1}\}
\]
on $U$ parameterized by $u \in S^{n-1}$ such that
\begin{enumerate}
\item $u \mapsto X_u(x)$ is a homeomorphism from $S^{n-1} \subset \mathbb R^n$ onto $S_F[x,0,1]$ for every $x \in U$;
\item $(x,u) \mapsto (y^1(x^1, \ldots, x^n,u), \ldots, y^n(x^1, \ldots, x^n,u))$ is continuous;
\item $(x,u) \mapsto (\frac{\partial y^1}{\partial x^i} (x^1, \ldots, x^n, u), \ldots, \frac{\partial y^n}{\partial x^i} (x^1, \ldots, x^n, u))$ is continuous for every $i=1, \ldots, n$.
\end{enumerate} 
We say that $F$ is of Pontryagin type on $M$ if it is of Pontryagin type at every $p \in M$. 
In this case $(M,F)$ is a Pontryagin type $C^0$-Finsler manifold.
\end{definition}

Remark \ref{coordinate changes trivializations} shows that  Definition \ref{horizontally smooth} doesn't depend on the coordinate system.
Remark \ref{mudanca de coordenadas - trivializacao local} goes a little bit further and it shows that Definition \ref{horizontally smooth} doesn't depend on the coordinate system on $TU$ corresponding to a local trivialization. 

\begin{remark}
\label{coordinate changes trivializations}
Let us see how the map
\[
(x^1, \ldots, x^n, u^1, \ldots, u^n) \mapsto (x^1, \ldots, x^n, y^1(x^1, \ldots, x^n,u), \ldots, y^n(x^1, \ldots, x^n,u))
\]
behaves under coordinate changes.
Let $\phi_1, \phi_2: U \subset M \rightarrow \mathbb R^n$, where $\phi_1 = (x^1, \ldots, x^n)$ and $\phi_2 = (\tilde x^1, \ldots, \tilde x^n)$. 
The natural coordinate system of $TU$ with respect to $\phi_2$ will be denoted by $(\tilde x^1, \ldots, \tilde x^n, \tilde y^1,$ $ \ldots, \tilde y^n)$.
The coordinate changes between them are given by
\[
\tilde x = \tilde x(x)
\]
and
\begin{equation}
\label{mudancafibravertical}
\tilde y^i = a_{ij}(x) y^i,
\end{equation}
where $a_{ij}$ are smooth functions on $U$. 
Equation (\ref{mudancafibravertical}) is due to the fact that the coordinate changes on each tangent space is linear.

If $X_u$ is a family of $C^1$ vector fields parameterized by $u$ given by
\[
(x, u)
\mapsto (x, y_1 (x, u), \ldots, y_n (x, u)),
\]
and
\[
(\tilde x, u)
\mapsto (\tilde x, \tilde y_1(\tilde x,u), \ldots, \tilde y_n(\tilde x,u)),
\]
then 
\begin{equation}
\label{mudanca vertical controle}
\tilde y_{i}(\tilde x,u) = a_{ij}(x(\tilde x)) y_j(x(\tilde x),u)
\end{equation}
due to (\ref{mudancafibravertical}).
From (\ref{mudanca vertical controle}), it is clear that if Conditions (1), (2) and (3) of Definition \ref{horizontally smooth} hold with respect to $\phi_1$, then they hold for every coordinate system on $U$. 
Therefore the concept of Pontryagin type $C^0$-Finsler structure doesn't depend on the choice of coordinate systems.
\end{remark}

\begin{remark}
\label{mudanca de coordenadas - trivializacao local}
Let $U$ be an open subset of the $C^0$-Finsler manifold $(M,F)$, $\phi=(x^1, \ldots,$  $x^n): U \rightarrow \mathbb R^n$ be a coordinate system and $\tau: TU \rightarrow U \times \mathbb R^n$ be a local trivialization of $U$. 
The smooth map
$\phi_\tau := (\phi \times \id)\circ \tau = (x^1, \ldots, x^n, \tilde y^1, \ldots, \tilde y^n): TU \rightarrow \phi(U) \times \mathbb R^n$ is a coordinate system on $TU$.
The coordinate changes from $\phi_{TU}$ to $(\phi \times \id)\circ \tau $ is given by
$(x^1, \ldots, x^n, y^1, \ldots, y^n) \mapsto (x^1, \ldots, x^n, \tilde y^1, \ldots, \tilde y^n)$, where (\ref{mudancafibravertical}) are in place because the coordinate changes are also fiberwise linear.
If we proceed as in Remark \ref{coordinate changes trivializations}, we have that the family of vector fields $X_u$ parameterized by $u$ is given by
\[
(x,u) \mapsto (x, \tilde y^1(x,u), \ldots, \tilde y^n(x,u)),
\] 
and it satisfies Conditions (1), (2) and (3) of Definition \ref{horizontally smooth} iff these conditions are also satisfied with respect to $\phi_{TU}$ .
From this remark it is straightforward that if $\phi$ and $\tilde \phi$ are coordinate systems on $U$ and $\tau$ and $\tilde \tau$ are local trivializations of $TU$, then $X_u$ satisfies Conditions (1), (2) and (3) of Definition \ref{horizontally smooth} with respect to the coordinate system $\phi_\tau$ iff these conditions are satisfied with respect to ${\tilde \phi}_{\tilde \tau}$.
Therefore Definition \ref{horizontally smooth} could be presented in a (a priori) more general format, but we defined in this way for the sake of simplicity.
\end{remark}

\section{The extended geodesic field}
\label{Hamiltonian formalism}

In this section we define the extended geodesic field $\mathcal E$ on $T^\ast M$ for Pontryagin type $C^0$-Finsler manifolds, which is obtained applying the PMP on the control system given in Definition \ref{horizontally smooth}.
Geodesics parametrized by arclength are projections of Pontryagin extremals of the differential inclusion $\gamma^\prime(t) = \mathcal E(\gamma(t))$ on $M$.
We will see some basic situations where $\mathcal E$ is a vector field as well as some of its properties.
In Theorem \ref{E generaliza G} we prove that $F_\ast .\mathcal E$ is a direct generalization of the geodesic spray of Finsler geometry, where $F_\ast:T^\ast M \rightarrow \mathbb R$ is the fiberwise dual norm of $F$.
We end this section making a comparison between $\mathcal E$ and $F_\ast . \mathcal E$ and the potential usefulness of each object.

The definition of $\mathcal E$ can be obtained from Definition \ref{horizontally smooth} following the calculations and remarks of \cite{Sachkov}, but we make the calculations explicitly here for the sake of completeness.

Let $(M^n,F)$ be a Pontryagin type $C^0$-Finsler manifold.
Let $\phi = (x^1, \ldots,$ $x^n): U \rightarrow \mathbb R^n$ be a coordinate system on $U$.
Define a control system on $M$ according to Definition \ref{horizontally smooth}.
We are interested to study geodesics on $U$.
Denote 
\[
X_u(x) = \sum_{i=1}^n f^i(x,u) \frac{\partial}{\partial x^i} = \left( f^1(x,u), \ldots, f^n(x,u) \right).
\]
The problem of minimizing the length of a path connecting two points in $(M,F)$ is a time minimizing problem of the control system
\begin{equation}
\label{sistema de controle 3}
x^\prime (t) = X_{u(t)},\;\; t \in [0,l],\;\; l\text{ is not fixed,}
\end{equation}
on $M$ because every $X_u$ is a unit vector field. 
Here the class of admissible controls $u(t)$ is the set of (bounded) measurable functions.

Set $\hat U = \mathbb R \times U$ and fix the coordinate system $\hat \phi=(x^0, x^1, \ldots, x^n)$ on $\hat U$, where $\mathbb R$ is parameterized by its canonical coordinate $x^0$.
Define the vector field 
\[
\hat X_u=\frac{\partial }{\partial x^0} + \sum_{i=1}^n f^i(x,u) \frac{\partial}{\partial x^i}
\]
on $\hat U$. 
It is immediate that $\hat X_u$ satisfy Conditions 1, 2 and 3 of PMP in any coordinate system on $\hat U$ due to Remark \ref{coordinate changes trivializations}.
Let $\hat \phi_{T^\ast \hat U}$ be the natural coordinate system on $T^\ast \hat U$ with respect to $\hat \phi$, which is given by
\[
\left( \hat \phi^{-1}(x^0, \ldots, x^n), \sum_{i=0}^n \alpha_i dx^i \right) \mapsto (x^0, \ldots, x^n, \alpha_0, \ldots, \alpha_n).
\] 

Let $\hat \theta$ be the tautological $1$-form on $T^\ast \hat U$ and for each $u \in C$ define the Hamiltonian $\hat H_u=\hat \theta(\hat X_u)$ on $T^\ast \hat U$ with respect to $u$.
Define $\hat {\mathcal M} = \sup_u \hat H_u$.
Let $\hat \omega=d\hat \theta$ be the canonical symplectic form on $T^\ast \hat U$.
The symplectic form induces an isomorphism between the tangent and cotangent bundles of $T^\ast \hat U$ given by $X \mapsto \hat \omega (\cdot, X)$.
If $X$ is a vector field on $T^*\hat U$, we denote the correspondent $1$-form by $X^\flat$ and if $\alpha$ is a $1$-form we denote the correspondent vector field by $\alpha^\sharp$.
The Hamiltonian vector field with respect to $u$ is given by $\vec {{\hat H}}_u=(d\hat H_u)^\sharp$.

Denote the tautological $1$-form and the canonical symplectic form on $T^\ast U$ by $\theta$ and $\omega$ respectively. 
Denote the objects of $T^\ast U$ corresponding to the respective objects of $T^\ast \hat U$ without the ``hat'' (For instance, $H_u$ is the Hamiltonian on $T^\ast U$ corresponding to $\hat H_u$).

The PMP for this problem states that if $u(t)$ and the respective solution $x(t)$ of (\ref{sistema de controle 3}) is a length minimizer, then there exists an absolutely continuous curve $\hat \gamma(t) = (t, x(t), \alpha_0(t), \alpha(t))$ on $T^\ast \hat U$ such that

\begin{enumerate}
\item $(\alpha_0(t), \alpha(t))$ $\neq 0$; 
\item $\hat \gamma^\prime(t) = {\vec {\hat H}}_{u(t)} (\hat \gamma(t))$ \text{ and }
\item $\hat H_{u(t)}(\hat \gamma(t))=\max_{u \in C}\hat H_u(\hat \gamma(t))$
\end{enumerate} 
almost everywhere. 
Moreover we have that
\[
\alpha_0(l) \leq 0 \text{ and }\hat{\mathcal M}(l)=0,
\]
at the terminal time $l$. In addition, if $u(t)$, $x(t)$ and $(\alpha_0(t), \alpha(t))$ determine an integral path of ${\vec {\hat H}}_u$ such that $\hat H_{u(t)}(\hat \gamma(t))=\max_{u \in C}\hat H_u(\hat \gamma(t))$ a.e., then $\alpha_0(t)$ and $\hat{\mathcal M}(t)$ are constant.

The following proposition is given in order to eliminate the term ``$\mathbb R$'' of $\hat U = \mathbb R \times U$. 

\begin{proposition}
\label{produto}
Let $M$ and $N$ be differentiable manifolds and let $C \subset \mathbb R^k$ be the  control set.
For every $u \in C$, define $C^1$ vector fields $X_u$ and $Y_u$ on $M$ and $N$ respectively such that Conditions 1, 2 and 3 of Definition \ref{horizontally smooth} are in place
(Here we don't impose any condition on the norm of the vector field because the $C^0$-Finsler structure isn't considered).
Define $H_{u,M} = \theta_M(X_u)$ and $H_{u,N} = \theta_N(X_u)$, where $\theta_M$ and $\theta_N$ are the tautological $1$-forms on $T^*M$ and $T^*N$ respectively.
Define $\vec{H}_{u,M}$ and $\vec{H}_{u,N}$ implicitly as $\omega_M (\cdot, \vec{H}_{u,M}) = dH_{u,M}$ and $\omega_N (\cdot, \vec{H}_{u,N}) = dH_{u,N}$ respectively, where $\omega_M$ and $\omega_N$ are the canonical symplectic forms on $M$ and $N$ respectively. 
If we consider the vector field $(X_u,Y_u)$ on $M \times N$ and $\vec{H}_{u,M\times N}$ is the respective Hamiltonian vector field on $T^*(M \times N)$, then $\vec{H}_{u,M \times N} = (\vec{H}_{u,M}, \vec{H}_{u,N})$.
\end{proposition}

\begin{proof}
Let $(U_M,(x^1,\ldots,x^m))$ and $(U_N,(\tilde{x}^1, \dots, \tilde{x}^n))$ be coordinate open subsets of $M$ and $N$ respectively and consider the coordinate open subset 
\[
(U_M \times U_N, (x^1,\ldots,x^m,\tilde{x}^1,\ldots, \tilde{x}^n))
\]
on $M \times N$. If we make calculations in coordinate systems, it is straightforward that the Hamiltonian vector field $\vec H_{u,M \times N}$ of $T^*(M \times N)$ is given by $(\vec H_{u,M}, \vec H_{u,N})$.
\end{proof}

Proposition \ref{produto} states that the projection of an integral curve of ${\vec {\hat H}}_{u(t)}$ to $M$ is also the projection of an integral curve of $\vec H_{u(t)}$. 
The key ingredient to define the extended geodesic field for Pontryagin type $C^0$-Finsler structure is the equality
\[
\hat H_{u(t)}(\hat \gamma(t)) = \max_{u \in C} \hat H_u(\hat \gamma(t)) \text{ a.e..}
\]
We have that
\[
\hat H_u = \hat \theta (\hat X_u) = \left( \alpha_0 dx_0, \alpha\right)\left( \frac{\partial}{\partial x_0}, X_u \right) = \alpha_0 + \alpha (X_u).
\]
But $\alpha_0(t)$ is a constant. 
Therefore for every $(x, \alpha) \in T^\ast U \backslash 0$, we must find a $u \in C$ such that $\alpha (X_u(x))$ is the maximum. But $S_F[x,0,1] = \{ X_u(x),u \in C \}$. Then
\[
\max_u \alpha(X_u(x)) = \max_{y \in S_F[x,0,1]} \alpha(y).
\]
Now we are in position to define the extended geodesic field.

\begin{definition}
\label{define geodesic field}
Let $(M,F)$ be a Pontryagin type $C^0$-Finsler manifold. 
The extended geodesic field of $(M,F)$ is the rule $\mathcal E$ that associates each $(x,\alpha) \in T^*M\backslash 0$ to the set $\mathcal E(x, \alpha) = \{\vec H_u(x, \alpha);u \in \mathcal C(x,\alpha)\}$, where $\mathcal C(x,\alpha) = \{u \in C;$ $ H_u(x,\alpha) = \max_{v\in C}H_v(x,\alpha) \}$.  Pontryagin extremals are absolutely continuous curves $\gamma:[a,b] \rightarrow T^\ast M\backslash 0$ which are solutions of the differential inclusion
\begin{align}
\label{differential inclusion}
\gamma^\prime(t)=\mathcal E(\gamma(t)).
\end{align}
\end{definition}

\begin{remark}
The definition of absolutely continuous curve on a differentiable manifold doesn't depend on the choice of the Riemannian metric (or $C^0$-Finsler structure) on $M$ because every pair of Riemannian metrics are locally Lipschitz equivalent. 
\end{remark}

The following theorem is fundamental for this work. 
Remember that every non-trivial minimizing curve on $(M,F)$ can be reparameterized by arclength (see \cite{Burago}). 

\begin{theorem}
\label{minimizing}
Let $(M,F)$ be a Pontryagin type $C^0$-Finsler manifold. Then every minimizing curve $x(t)$ of $(M,F)$ parameterized by arclength is the projection of a Pontryagin extremal $(x(t),\alpha(t))$.
Consequently the Hamiltonian $H_{u(t)}(x(t),\alpha(t))$ is constant.
\end{theorem}

\begin{proof}
For the proof of this theorem, it is enough to find an admissible control $u(t)$ such that $X_{u(t)}(x(t)) = x^\prime (t)$ due to the PMP  and the definition of $\mathcal E$.
Since the measurability of $t \mapsto u(t)$ is a local property, we can prove it in an open subset $U$ which is compactly embedded in a coordinate open subset $(U^\prime, \phi=(x^1, \ldots, x^n))$.
We denote the unit fiber bundle of $(U,F)$ by $SU$. 

We represent the map $t \mapsto u(t)$ as a composition of three maps.
The first map is the projection $\pi : U \times S^{n-1} \rightarrow S^{n-1}$, where $S^{n-1}$ is the Euclidean unit sphere in $\mathbb R^n$. 
The last one is $t\mapsto (x(t), x^\prime(t))$, which we denote by $\eta_1:[a,b] \rightarrow SU$.
Observe that $\eta_1$ is measurable because $x(t)$ is locally Lipschitz (see \cite{Royden}).
The second map is contructed as follows: 
We know from Items (1) and (2) of Definition \ref{horizontally smooth} that the map $(x,u)\mapsto X_u(x)$ is a continuous bijection from the compact space $\overline{U \times S^{n-1}}$ onto the Hausdorff space $\overline{SU}$. 
Therefore it is a homeomorphism.
The second map is defined as its inverse map $\eta_2: SU \rightarrow U \times S^{n-1}$.
Finally observe that $t \mapsto u(t)$ is measurable because it is given by the composition of measurable functions $\pi \circ \eta_2 \circ \eta_1$, what settles the theorem.
\end{proof}

{\em From now on the Einstein convention for the summation of indices is in place because the indices will not vary from $0$ to $n$ anymore.}

Direct calculations in natural coordinate systems $\phi_{TU} = (x^1, \ldots,$ $x^n, y^1, \ldots, y^n)$ and $\phi_{T^\ast U} = (x^1, \ldots, x^n, \alpha_1,$ $\ldots, \alpha_n)$ give
\[
\theta = \alpha_i dx^i,\hspace{4mm} \omega = d\alpha_i \wedge dx^i, \hspace{4mm}
H_u = \alpha_i f^i(x,u),
\]
\begin{equation}
\label{campohamiltonianocoordenadas2}
\vec H_u = f^i(x,u) \frac{\partial}{\partial x^i} - \alpha_j \frac{\partial f^j}{\partial x^i}(x,u) \frac{\partial}{\partial \alpha_i}
\end{equation}
and

\begin{equation}
\label{define-extended}
\mathcal E(x,\alpha) = f^i(x,u(x,\alpha)) \frac{\partial}{\partial x^i} - \alpha_j \frac{\partial f^j}{\partial x^i}(x,u(x,\alpha)) \frac{\partial}{\partial \alpha_i},
\end{equation}
where $u(x,\alpha) \in \mathcal C(x,\alpha)$.

The following propositions are straightforward from the definition of $\mathcal E$.

\begin{proposition}
Let $(M,F)$ be a Pontryagin type $C^0$-Finsler manifold. If $F(x,\cdot)$ is strictly convex for every $x\in M$, then $\mathcal{E}$ is a vector field on $T^*M \backslash 0$.
\end{proposition}

\begin{proposition}
If $u(x,\alpha)$ is a continuous function and $\mathcal{E}$ is a vector field, then $\mathcal E$ is continuous.
\end{proposition}

\begin{proposition}
If $\frac{\partial f^j}{\partial x^i}$ and $u(x,\alpha)$ are locally Lipschitz functions, then $\mathcal E$ is a locally Lipschitz vector field.
\end{proposition}

In the Riemannian case we have a diffeomorphism between $TM$ and $T^\ast M$ given by the Legendre transform $(x,y) \mapsto (x,y^\flat)$, where the musical isomorphism is defined with respect to  the Riemannian metric restricted to each tangent space. 

In the Finsler case we have the following situation: 
Let $F_\ast : T^\ast M \rightarrow \mathbb R$ be the fiberwise dual norm of $F: TM \rightarrow \mathbb R$.
Then $F_\ast$ is a Finsler structure on $T^\ast M$.
The Legendre transform $TM\backslash 0 \rightarrow T^\ast M \backslash 0$ is a diffeomorphism given by 
\begin{equation}
\label{legendre-Finsler}
(x^1, \ldots, x^n, y^1, \ldots, y^n) \mapsto (x^1, \ldots, x^n,  g_{1j(x,y)} y^j, \ldots, g_{nj(x,y)}y^j)
\end{equation}
and its inverse transform is given by
\begin{equation}
\label{inversa-legendre}
(x^1, \ldots, x^n, \alpha_1, \ldots, \alpha_n) \mapsto (x^1, \ldots, x^n, g^{1j(x,\alpha)} \alpha_j, \ldots, g^{nj(x,\alpha)}\alpha_j),
\end{equation}
where $g_{ij(x,y)}$ are the components of the fundamental tensor of $F$ at $(x,y)$ and $g^{ij(x,\alpha)}$ are the components of the fundamental tensor of $F_\ast$ at $(x,\alpha)$.
Notice that $g^{ij(x,\alpha)} = g^{ij(x,\alpha^\sharp)}$ holds, where the musical isomorphism is given with respect to the fundamental tensor of $F$ 
(see Section 14.8 of \cite{BaoChernShen} for vector spaces endowed with Minkowski norms. 
The extension for $TM\backslash 0$ is straightforward). 

Now we calculate the geodesic spray $\mathcal G$ on $T^\ast M\backslash 0$ for the Finsler case.
Its proof is an adaptation of the Riemannian case. 

\begin{theorem}
\label{caso Finsler} 
Let $(M,F)$ be a Finsler manifold. 
Consider the identification of $TM\backslash 0$ and $T^\ast M \backslash 0$ given by the Legendre transform. 
Then the geodesic spray on $T^\ast M\backslash 0$ through this identification is given by
\begin{equation}
\label{spray no fibrado cotangente}
\mathcal G(x,\alpha) = g^{ik}\alpha_k\frac{\partial}{\partial x^i}-\frac{1}{2}\alpha_j\alpha_k\frac{\partial g^{jk}}{\partial x^i}\frac{\partial}{\partial \alpha_i},
\end{equation}
where $g^{ij}$ are the coefficients of the fundamental tensor of $F_\ast$.
\end{theorem}

\begin{proof}

Let $(M,F)$ be a Finsler manifold. Let $\phi=(x^1, \ldots, x^n):U \subset M \rightarrow \mathbb R^n$ be a coordinate system on a open subset $U$ of $M$ and let $(x^1, \ldots, x^n, y^1, \ldots, y^n)$ and $(x^1, \ldots, y^n, \alpha_1, \ldots, \alpha_n)$ be the natural coordinates on $TU$ and $T^\ast U$ respectively.

There is a subtlety in this proof. Although the coordinate functions $(x^1, \ldots,$ $x^n)$ of $TU$ and $T^\ast U$ are identified, the coordinate vector fields $\partial / \partial x^i$ on these bundles aren't the same in general. 
In order to make this distinction, $\partial /\partial \hat x^i$ will represent the coordinate vector field on $TU$ and $\partial / \partial x^i$ will represent the coordinate vector field on $T^\ast U$.  
Observe that
\begin{equation}
\label{relacaoxichapeuxi}
\frac{\partial}{\partial \hat x^i} = \frac{\partial}{\partial x^i} - \frac{\partial g^{lm}}{\partial x^i}\alpha_m
\frac{\partial}{\partial y^l}
\end{equation}
due to (\ref{inversa-legendre}).

The geodesic equation of a Finsler manifold $(M,F)$ is given by
\begin{equation}
\label{geodesic equation}
\frac{d^2 x^i}{dt^2} + \gamma^i{}_{jk} \frac{dx^j}{dt} \frac{dx^k}{dt}=0,
\end{equation}
where $\gamma^i{}_{jk}$ is given by (\ref{Christoffel}).
Observe that
\begin{equation}
\label{variacao de x}
\frac{dx^i}{dt}=y^i = g^{ik}\alpha_k
\end{equation}
(We can consider $g^{ij}$ as the inverse of the fundamental tensor of $F$ at $(x,y)$ or else the fundamental tensor of $F_\ast$ at $(x,y^\flat)$ because $g^{ij(x,y)} = g^{ij(x,y^\flat)}$).
Moreover
\begin{align}
\frac{d^2 x^i}{dt^2} 
& = \frac{d}{dt}y^i 
= \frac{d}{dt}g^{ij}\alpha_j 
= \frac{dg^{ij}}{dt}\alpha_j + g^{ij}\frac{d\alpha_j}{dt}  = \frac{\partial g^{ij}}{\partial \hat x^k} y^k \alpha_j + \frac{\partial g^{ij}}{\partial y^k} \frac{dy^k}{dt} g_{js} y^s + g^{ij}\frac{d\alpha_j}{dt}. \label{derivada segunda x}
\end{align}
The middle term of the right-hand side of (\ref{derivada segunda x}) is zero.
In fact, this is due to (\ref{cartan e y zera}) and
\begin{equation}
\label{derivada gij e inversa}
\frac{\partial g^{ij}}{\partial y^k} = -g^{il}g^{mj}\frac{\partial g_{lm}}{\partial y^k}.
\end{equation}
The second term of the left-hand side of (\ref{geodesic equation}) is given by
\begin{align}
\gamma^i{}_{jk}y^jy^k 
& = \gamma^i{}_{jk}g^{jl}\alpha_l g^{km}\alpha_m 
=\frac{1}{2}g^{is}\left(\frac{\partial g_{sj}}{\partial \hat x^k} + \frac{\partial g_{sk}}{\partial \hat x^j} - \frac{\partial g_{jk}}{\partial \hat x^s} \right)g^{jl} g^{km}\alpha_l \alpha_m \nonumber \\ 
& = - \frac{\partial g^{ij}}{\partial \hat x^k} y^k \alpha_j + \frac{1}{2}\frac{\partial g^{lm}}{\partial \hat x^s}g^{is}\alpha_l \alpha_m, \label{termo gammaijyjyk}
\end{align}
where the last equation is due to (\ref{derivada gij e inversa}) with $\partial / \partial y^k$ replaced by $\partial / \partial \hat x^k$.
Replacing (\ref{derivada segunda x}) and (\ref{termo gammaijyjyk}) in (\ref{geodesic equation}) we get
\begin{equation}
\label{quase-spray-cotangente}
\frac{d\alpha_i}{dt} = -\frac{1}{2}\frac{\partial g^{jk}}{\partial \hat x^i}\alpha_j \alpha_k.
\end{equation}
Finally we use (\ref{relacaoxichapeuxi}) on (\ref{quase-spray-cotangente}), and the term with vertical derivative is zero due to  (\ref{cartan e y zera}) and (\ref{derivada gij e inversa}), what settles the theorem.
\end{proof}

Observe that in the Finsler case we have that
\begin{equation}
\label{expressao-xu-Finsler}
X_{u(x,\alpha)} = f^i(x,u(x,\alpha)) \frac{\partial}{\partial x^i} = \frac{g^{ik(x,\alpha)}\alpha_k}{\sqrt{\alpha_l g^{lm(x,\alpha)} \alpha_m}} \frac{\partial}{\partial x^i}
\end{equation}
because the right-hand side of (\ref{expressao-xu-Finsler}) is a unit vector and
\[
\alpha \left( \frac{g^{ik(x,\alpha)}\alpha_k}{\sqrt{\alpha_l g^{lm(x,\alpha)} \alpha_m}} \frac{\partial}{\partial x^i}\right) = F(\alpha^\sharp)
= F_\ast (\alpha) = \max_{y \in S_F[x,0,1]} \alpha (y)
\]
(for the second equality, see \cite{BaoChernShen}).
It follows from (\ref{define-extended}), (\ref{spray no fibrado cotangente}) and the second equality of (\ref{expressao-xu-Finsler}) that $\mathcal G(x,\alpha) = F_\ast (x,\alpha) \mathcal E(x,\alpha)$. 
Therefore $F_\ast . \mathcal E$ is a generalization of $\mathcal G$ for Pontryagin type $C^0$-Finsler manifolds.
Let us compare the integral curves of $F_\ast . \mathcal E$ and $\mathcal E$.

If $(x(t),\alpha(t))$ be a integral curve of $\mathcal E$ such that $x(t)$ is a minimizing path, then 
\[
F_\ast (x(t),\alpha(t)) = \max\limits_{v \in S_F[x(t),0,1]} [\alpha(t)] (v) = \mathcal M(x(t),\alpha(t))
\]
is constant along the curve. 
Moreover if $c_1>0$, then $(x(t),c_1\alpha(t))$ is also an integral curve of $\mathcal E$.
Therefore we can choose $\alpha(t)$ with unit norm along the curve. 
In this case, if $c_2>0$, then $(x(c_2t),c_2\alpha (c_2t))$ is an integral curve of $F_\ast (x,\alpha) . \mathcal E$.

Reciprocally let $\gamma(t) = (x(t),\alpha(t))$ defined on $(-\epsilon, \epsilon)$ be an integral curve of $F_\ast .\mathcal E$ such that $x(t)$ is a minimizing path.
We claim that there exist $c>0$ such that $(x(t/c), y(t/c))$ is an integral curve of $\mathcal E$.
In fact, consider the initial value problem
\begin{equation}
\label{equacao-diferencial-zeta}
\zeta^\prime(t) = \frac{1}{F_\ast\circ \gamma } (\zeta (t)); \hspace{3mm} \zeta(0)=0.
\end{equation}
It has a unique strictly increasing solution $\zeta(t)$.
It is straightforward that 
\[
(\tilde x(t), \tilde \alpha(t)) := (x(\zeta(t)), \alpha (\zeta(t)))
\] 
is an integral curve of $\mathcal E$. 
If follows that 
\[
F_\ast (x(\zeta(t)), \alpha (\zeta(t))) = F_\ast (\tilde x(t),\tilde \alpha(t))
\]
is constant, what implies that $\zeta^\prime (t)$ is a positive constant $c$ due to (\ref{equacao-diferencial-zeta}).
Therefore $(x(t),\alpha(t)) = (\tilde x(t/c),\tilde \alpha (t/c))$.

The relationship between $F_\ast . \mathcal E$ and $\mathcal E$ can be summarized as follows: 
if the minimizing path $x(t)$ is the projection of an integral curve of $\mathcal E$ on $M$ and $c>0$, then $x(ct)$ is the projection of an integral curve of $F_\ast . \mathcal E$. 
Reciprocally if the minimizing path $x(t)$ is the projection of an integral curve of $F_\ast \mathcal E$ on $M$, then it has constant speed and its reparameterization by arclength is the projection of an integral curve of $\mathcal E$. 
Therefore $F_\ast . \mathcal E$ represents all minimizing paths with constant speed and $\mathcal E$ represents all minimizing paths parameterized by arclength.
We have proved the following theorem.

\begin{theorem}
\label{E generaliza G} The correspondences $\mathcal E$ and $F_\ast . \mathcal E$ are generalizations of the geodesic spray $\mathcal G$ of Finsler geometry.
\end{theorem}

It isn't clear whether $F_\ast .\mathcal E$ or $\mathcal E$ will be more useful for the theory of Pontryagin type $C^0$-Finsler manifolds.
The structure $F_\ast . \mathcal E$ is a direct generalization of the geodesic spray of Finsler geometry and it can have more potential to bring geometric objects of Finsler geometry to $C^0$-Finsler geometry. On the other hand, $\mathcal E$ is horizontally bounded and it can be more suitable for direct calculations and the study of the existence of solutions of $\mathcal E$. In this work we use $\mathcal E$ because it is more convenient for our purposes. 

\section{Asymmetric norms and its dual asymmetric norm}
\label{asymmetric norms}

Let $V$ be a finite dimensional real vector space endowed with an asymmetric norm $F$. 
In this section we study some relationships between $(V,F)$ and its dual space $(V^\ast, F_\ast)$, where $F_\ast$ is the dual asymmetric norm of $F$.

\begin{definition}
\label{conjugadoconvexo}
Let $V$ be a vector space and $f:V\rightarrow \mathbb{R}$ be a convex function. The convex conjugate or Fenchel transformation of $f$ is the function $f^*:V^*\rightarrow \mathbb{R}$ defined by
\begin{equation*}
f^*(\alpha)=\sup_{y\in V }\{\alpha(y)-f(y) \}.
\end{equation*}
\end{definition}

\begin{definition}
\label{Normadual}
Let $F$ be an asymmetric norm on $V$. The dual asymmetric norm $F_*:V^*\rightarrow \mathbb{R}$ of $F$ is defined by
\begin{equation*}
F_*(\alpha)=\sup_{F(y)\leq{1}}\alpha(y).
\end{equation*}
\end{definition}

If $\alpha\in V^*$ and $y\in V \backslash \{0\}$, then

\begin{equation*}
\alpha(y)=F(y) \alpha \left(\frac{y}{F(y)} \right)\leq F(y)F_*(\alpha)
\end{equation*}
and the inequality
\begin{equation}
\label{desigualdade h(x) F(x)F_*(h)}
\alpha(y)\leq F(y)F_*(\alpha)
\end{equation}
holds for every $y\in V$ and $\alpha\in V^*$.

Given a point $y\in V$, we will denote the set of functional supports of $F^2$ at $y$ by $\partial F^2(y)$, that is, the set of $\alpha\in V^*$ such that
\begin{equation*}
F^2(z)\geq F^2(y)+\alpha(z-y)\ \ \ \ \ \ \forall z\in V.
\end{equation*}
The set $\partial F^2(y)$ is called subdifferential of $F^2$ at $y$.

The next lemma relates the convex conjugate of an asymmetric norm to the dual asymmetric norm.

\begin{lemma}
\label{conjugadodual}
If $F$ is an asymmetric norm on $V$, then ${F^2}^*=\frac{1}{4}{F_*}^2$.
\end{lemma}
\begin{proof}
Let us show first that ${F^2}^*\leq\frac{1}{4}{F_*}^2$. From (\ref{desigualdade h(x) F(x)F_*(h)}) we have $\alpha(y)-F^2(y)\leq F(y)F_*(\alpha)-F^2(y)$ for every $y\in V$. Note that $F(y)F_*(\alpha)-F^2(y)$ is a 
quadratic function on $F(y)$ and it reaches its maximum at $F(y)=\frac{1}{2}F_*(\alpha)$. Therefore
\begin{align*}
\alpha(y) -F^2(y)&\leq F(y)F_*(\alpha)-F^2(y) \leq \frac{1}{4}{F_*}^2(\alpha)
\end{align*}
for every $y\in V$.
	
Now, let's show the opposite inequality. Since $F_*(\alpha)=\sup_{F(y)\leq{1}}\alpha(y)$, there exists $y'\in S_F[0,1]$ such that $F_*(\alpha)=\alpha(y')$.
Set $y=\frac{1}{2}F_\ast (\alpha) y^\prime$.
Then $F(y)=\frac{1}{2}F_\ast (\alpha)$ and 
\begin{align*}
\alpha(y)-F^2(y) & =\frac{1}{2}{F_*}^2(\alpha)-\frac{1}{4}{F_*}^2(\alpha) 
=\frac{1}{4}{F_*}^2(\alpha),
\end{align*}
what implies
\begin{equation*}
{F^2}^*(\alpha)=\sup_{y\in V }\{\alpha(y)-F^2(y) \}\geq \frac{1}{4}{F_*}^2(\alpha).
\end{equation*}
Therefore, ${F^2}^*(\alpha)=\frac{1}{4}{F_*}^2(\alpha)$.
\end{proof}

If $F$ is a strictly convex asymmetric norm, then ${F^2}^*$ is differentiable (see \cite{RockafellarTyrrell}) and ${F_*}^2$ is differentiable due to Lemma \ref{conjugadodual}. 
\textit{We will consider that $F$ is strictly convex until the end of this section}. 

\begin{remark}
\label{diferencial em v estrela}
The differential $dF^2_\ast: V^\ast \times V^\ast \rightarrow \mathbb R$ is naturally identified with $dF^2_\ast: V^\ast \rightarrow V$. 
From now on we consider the latter.
\end{remark}

If $f: V \rightarrow \mathbb R$ is a strictly convex function, then $\partial f^*$ is the inverse of $\partial f$ in the sense of multivalued mapping, that is, $y\in\partial f^*(y^*)$ if and only if $y^*\in\partial f(y)$ (see \cite{RockafellarTyrrell}). Therefore $(\partial F^2)^{-1}=\partial {F^2}^*$ and

\begin{equation}
\label{inversasubdiferencial}
(\partial F^2)^{-1}=\frac{\partial{F_*}^2}{4}=\frac{1}{4}d{F_*}^2
\end{equation}
due to Lemma \ref{conjugadodual}.

Given $\alpha\in V^*$, there is a unique $y\in V$ such that $F^2(z)\geq F^2(y)+\alpha(z - y)$ for every $z\in V$, that is, $\alpha\in\partial F^2(y)$. 
The graph of the function $\varphi$ defined by $\varphi (z)=F^2(y)+\alpha(z-y)$ is tangent to the graph of $F^2$ at $y$, what implies that the affine subspace $\{z\in V; \varphi (z)=F^2(y)\}$ of $V$ is tangent to the level set $S_F[0,F(y)]$ at $y$. 
From this tangency and the strict convexity of $B_F[0,F(y)]$, it follows that $\varphi (z)\leq F^2(y)$ for every $z\in S_F[0,F(y)]$ and the equality holds iff $z=y$. 
Therefore, we can conclude that $y$ is the unique point in $S_F[0,F(y)]$ that maximizes $\alpha$ (and $\varphi$) and $\frac{y}{F(y)}$ is the unique point that  maximizes $\alpha$ in $S_F[0,1]$. 
Therefore $\frac{(\partial F^2)^{-1}(\alpha)}{F((\partial F^2)^{-1}(\alpha))}$ is the unique point in $S_F[0,1]$ that maximizes $\alpha$ and this point can be written as

\begin{equation*}
\frac{\frac{d{F_*}^2(\alpha)}{4}}{F(\frac{d{F_*}^2(\alpha)}{4})}=\frac{d{F_*}^2(\alpha)}{F(d{F_*}^2(\alpha))}
\end{equation*} 
due to (\ref{inversasubdiferencial}).
This proves the following proposition:

\begin{proposition}
\label{unico vetor que maximiza o funcional caso estritamente convexo}
If $F:V\rightarrow\mathbb{R}$ is a strictly convex asymmetric norm and $\alpha\in V^*$, then
\begin{equation}
\frac{d{F_*}^2(\alpha)}{F(d{F_*}^2(\alpha))}
\end{equation}
is the unique point in $S_F[0,1]$ that maximizes $\alpha$.
\end{proposition}

Here is an important point: note that the search for a locally Lipschitz application $\alpha\mapsto \frac{y}{F(y)}$ goes through the study of $d{F_*}^2$.

The next lemma is a consequence of Euler's theorem and it shows how the duality between $F$ and $F_*$ behaves.
\begin{lemma}
\label{igualdadefundamental}
$F(d{F_*}^2(\alpha))=2F_*(\alpha)$.
\end{lemma}
\begin{proof}
Since ${F_*}^2$ is homogeneous of degree 2 and its radial derivative at $\alpha$ is given by $\alpha(d{F_*}^2(\alpha))$, we have that $\alpha(d{F_*}^2(\alpha))=2{F_*}^2(\alpha)$ due to the Euler's theorem. Thus 
\begin{align*}
F_*(\alpha) &= \alpha\left(\frac{d{F_*}^2(\alpha)}{F(d{F_*}^2(\alpha))}\right)
=\frac{1}{F(d{F_*}^2(\alpha))}\alpha(d{F_*}^2(\alpha))
=\frac{2{F_*}^2(\alpha)}{F(d{F_*}^2(\alpha))}
\end{align*}
and $F(d{F_*}^2(\alpha))=2F_*(\alpha)$.
\end{proof}

\begin{remark}
\label{dfasterisco-legendre}
The application $dF^2_\ast$ is essentially the inverse of the Legendre transform (see (\ref{inversasubdiferencial})), which is much more well behaved than $\partial F^2$.
Therefore it is ``much easier to access'' $TM$ from a geometric structure on $T^\ast M$ than to proceed in the opposite direction (compare with Remark \ref{conclusoes 1} and Remark \ref{conclusoes 2}).
\end{remark}

In order to have a better control over $d{F_*}^2:V^*\rightarrow V$, we define a class of asymmetric norms that is more restrictive than the strictly convex ones.
\begin{definition}
\label{definicaofortementeconvexa}
Let $\check F$ be an asymmetric norm on $V$. We say that an asymmetric norm $F$ is strongly convex with respect to $\check F$ if 
\begin{equation}
F^2(z)\geq F^2(y)+ \alpha(z-y) +\check F^2(z-y)
\end{equation}
for every $y,z\in V$ and $\alpha\in\partial F^2(y)$.
\end{definition}

\begin{remark}
\label{fortemente convexo nao depende f til}
All asymmetric norms on vector spaces are equivalent, that is, if $\check F$ and $\hat F$ are asymmetric norms on $V$, then there exist $c_1,c_2>0$ such that
\[
c_1 \check F (y) \leq \hat F (y) \leq c_2 \check F(y) 
\] 
for every $y \in V$. 
Therefore if $F$ is strongly convex with respect to $\check F$, then $F$ will be strongly convex with respect to a positive multiple of $\hat F$.
\end{remark}

Whenever it is clear from the context, we will omit the asymmetric norm with respect to which $F$ is strongly convex.
\begin{theorem}
\label{diferencialduallipschitz}
If $F$ is a strongly convex asymmetric norm on $V$ with respect to
$\check F$, then the application $d{F_*}^2:V^*\rightarrow V$ is Lipschitz.
\end{theorem}
\begin{proof}
First of all notice that the fact that $dF_\ast^2$ is Lipschitz doesn't depend on the asymmetric norms we consider on $V^\ast$ and $V$ due to Remark \ref{fortemente convexo nao depende f til}.
Therefore we can fix Euclidean norms $\Vert \cdot \Vert$ and $\Vert \cdot \Vert_\ast$  on $V$ and $V^\ast$ respectively and replace $\check F$ by $c\Vert \cdot \Vert$, with $c>0$.
Consider $\alpha_1,\alpha_2\in V^*$. By (\ref{inversasubdiferencial}) and Definition \ref{definicaofortementeconvexa}, we get
\begin{align}
\label{diferencialduallipschitz 1}
F^2(\frac{1}{4}d{F_*}^2(\alpha_2)) &\geq F^2(\frac{1}{4}d{F_*}^2(\alpha_1))+ \alpha_1(\frac{1}{4}d{F_*}^2(\alpha_2)-\frac{1}{4}d{F_*}^2(\alpha_1))+c^2\left \Vert \frac{1}{4}d{F_*}^2(\alpha_2)-\frac{1}{4}d{F_*}^2(\alpha_1) \right \Vert^2
\end{align}
and 
\begin{align}
\label{diferencialduallipschitz 2}
F^2(\frac{1}{4}d{F_*}^2(\alpha_1)) &\geq F^2(\frac{1}{4}d{F_*}^2(\alpha_2))+\alpha_2(\frac{1}{4}d{F_*}^2(\alpha_1)-\frac{1}{4}d{F_*}^2(\alpha_2)) + c^2\left \Vert \frac{1}{4}d{F_*}^2(\alpha_1)-\frac{1}{4}d{F_*}^2(\alpha_2) \right \Vert^2.
\end{align}
Summing up (\ref{diferencialduallipschitz 1}) and (\ref{diferencialduallipschitz 2}), we have
\begin{equation*}
2c^2\left \Vert \frac{1}{4}d{F_*}^2(\alpha_1)-\frac{1}{4}d{F_*}^2(\alpha_2) \right \Vert^2\leq (\alpha_1-\alpha_2)(\frac{1}{4}d{F_*}^2(\alpha_1)-\frac{1}{4}d{F_*}^2(\alpha_2))
\end{equation*}
and 
\begin{equation*}
2c^2\left \Vert \frac{1}{4}d{F_*}^2(\alpha_1)-\frac{1}{4}d{F_*}^2(\alpha_2) \right \Vert^2\leq \left \Vert \frac{1}{4}d{F_*}^2(\alpha_1)-\frac{1}{4}d{F_*}^2(\alpha_2) \right \Vert \left \Vert \alpha_1-\alpha_2 \right \Vert_* ,
\end{equation*}
what implies
\begin{equation*}
\left \Vert d{F_*}^2(\alpha_1)-d{F_*}^2(\alpha_2) \right \Vert\leq \frac{2}{c^2}\left \Vert \alpha_1-\alpha_2 \right \Vert_*.
\end{equation*}
\end{proof}

\begin{remark} 
\label{differencial}
Theorem \ref{diferencialduallipschitz} and its proof also work if we consider a family of strongly convex asymmetric norms $\{F_\nu\}_{\nu \in \Lambda}$, where every $F_\nu$ is strongly convex with respect to the same $\check F$.
We get
\[
\Vert d(F_{\nu})^2_\ast (\alpha_1) - d(F_{\nu})^2_\ast (\alpha_2) \Vert \leq \frac{2}{c^2}\Vert \alpha_1 - \alpha_2\Vert_\ast
\]
for every $\nu \in \Lambda$ and $\alpha_1, \alpha_2 \in V^\ast$.
\end{remark}

\section{Horizontally $C^1$ family of asymmetric norms}
\label{Horizontally C1 family of asymmetric norms subsection}
In this section, $\Vert \cdot \Vert$ is the canonical norm on $\mathbb R^n $ and $\Vert \cdot \Vert_\ast$ is its dual norm on $\mathbb R^{n\ast}$.

In the previous section we saw a condition for the application $d{F_*}^2:V^*\rightarrow V$ to be Lipschitz. Here we will analyze $d{F^2_*}$ for a family of asymmetric norms defined as follows.
\begin{definition}
\label{familia horizontalmente c1 de normas Rn}
Let $U\subset\mathbb{R}^n$ be an open subset. We say that a continuous function $F:U\times\mathbb{R}^n\rightarrow \mathbb{R}$ is a horizontally $C^1$ family of asymmetric norms if
\begin{enumerate}
\item $F(x,\cdot):\mathbb{R}^n\rightarrow \mathbb{R}$ is an asymmetric norm for every $x\in U$
and
\item $(x,y)\mapsto \frac{\partial F}{\partial x_i}(x,y)$ is continuous for every $i=1,\ldots,n$.
\end{enumerate}
\end{definition}

\begin{remark} 
A horizontally $C^1$ family of asymmetric norms is a particular instance of $C^1$-partially smooth $C^0$-Finsler structure defined in \cite{MatveevTroyanov}. 
The difference is that in the latter definition the horizontal partial derivatives don't need to be continuous.  
\end{remark}

\begin{definition}
Let $F$ be a horizontally $C^1$ family of asymmetric norms. We will denote the subdifferential of $F(x,\cdot)$ at $y$ by $\partial F(x,y)$.
\end{definition}

\begin{definition}
Let $F$ be a horizontally $C^1$ family of asymmetric norms.
The horizontal differential $d_hF: U \times \mathbb R^n \times \mathbb R^n \rightarrow \mathbb R$ of $F$ is defined by
\[
d_hF(x,w,y)=\lim_{t \rightarrow 0}\frac{F(x+tw,y)-F(x,y)}{t}.
\]
\end{definition}

\begin{remark}
\label{explica DhF}
$d_hF$ is naturally identified with the map
$\mu:U \times \mathbb R^n \rightarrow \mathbb R^{n\ast}$ 
defined by $\mu (x,y)=d_hF(x,\cdot,y)$.
From now on we denote $\mu$ by $d_hF$, which is given by
\[
d_hF(x,y) = \frac{\partial F(x,y)}{\partial x^i}dx^i
\]
in the natural coordinate system.
Notice that
\[
\left\Vert d_hF \right\Vert_\ast  = \left\Vert \left( \frac{\partial F}{\partial x^1}, \ldots,\frac{\partial F}{\partial x^n} \right) \right\Vert_\ast,
\]
and $d_hF$ is locally Lipschitz iff $\frac{\partial F}{\partial x^i}$ is locally Lipschitz for every $i=1, \ldots n$.
\end{remark}

We will denote for each $x\in U$, the dual asymmetric norm of $F(x,\cdot)$ by $F_*(x,\cdot)$, what gives us a family of dual asymmetric norms
\begin{equation}
\label{familia de normas duais}
F_*:U\times\mathbb{R}^{n*}\rightarrow\mathbb{R}.
\end{equation}
If $F(x,\cdot)$ is strictly convex, we have its ``vertical'' differentials
\begin{equation}
\label{familia de subdeferenciais}
d_vF_\ast: U \times \mathbb R^{n\ast} \rightarrow \mathbb R^n
\end{equation}
(compare with Remark \ref{diferencial em v estrela}).
As the definition above suggests, we will study the variation of (\ref{familia de subdeferenciais}). The study will be done separately in what is called horizontal and vertical direction, i.e. along $U$ and $\mathbb{R}^{n*}$ respectively. {\em We will use the notation $U\subset\subset \widetilde U$ to state that $U$ is compactly embedded in $\widetilde U$, that is, the closure $\bar U$ of $U$ is compact and $\bar{U}\subset \widetilde U$.}
Initially, we will see some results for (\ref{familia de normas duais}).

\begin{proposition}
\label{condicao de Lipschitz horizontal}
Let $U,\widetilde U\subset\mathbb{R}^n$ be open subsets such that $U$ is convex and $U\subset\subset \widetilde U$. If $F:\widetilde U\times\mathbb{R}^n\rightarrow \mathbb{R}$ is a horizontally $C^1$ family of asymmetric norms, then there exists $C_1>0$ such that
\begin{equation*}
\vert F_*(x_1,\alpha)-F_*(x_2,\alpha)\vert \leq C_1\left \Vert \alpha \right \Vert_* \left \Vert x_1-x_2 \right \Vert
\end{equation*}
for every $\alpha\in\mathbb{R}^{n*}$ and $x_1,x_2\in U$.
\end{proposition}
\begin{proof}
Let $(x,y)\mapsto \frac{y}{F(x,y)}$ be the projection onto the spheres of $F$. Since $F$ is horizontally $C^1$, it follows from the mean value theorem that given $x_1,x_2\in U$ and $y\in\mathbb{R}^n$ there exists $\theta\in(0,1)$ such that
\begin{equation}
\label{teorema do valor medio em xsobreF_p(x)}
\left \Vert \frac{y}{F(x_1, y)}-\frac{y}{F(x_2,y)}\right \Vert \leq  \frac{\Vert y \Vert \left \Vert (d_hF)(x_1+\theta(x_2-x_1),y)\right \Vert}{F^2(x_1+\theta(x_2-x_1),y)}\left \Vert x_1-x_2\right \Vert
\end{equation}
Therefore,
\begin{align*}
\vert F_*(x_1,\alpha)-F_*(x_2,\alpha)\vert
 = & \left\vert \sup_{y\in\mathbb{R}^n\backslash{0}}\alpha\left(\frac{y}{F(x_1,y)}\right)-\sup_{y\in\mathbb{R}^n}\alpha\left(\frac{y}{F(x_2,y)}\right)\right\vert \\
& = \left \Vert \alpha \right \Vert_* \sup_{y\in\mathbb{R}^n\backslash{0}} \left \Vert \frac{y}{F(x_1,y)}-\frac{y}{F(x_2,y)} \right \Vert \\
\leq & \left \Vert \alpha \right \Vert_* \sup_{y\in S^{n-1}} \frac{\left \Vert y \right \Vert\left \Vert (d_hF)(x_1+\theta(x_2-x_1),y)\right \Vert}{F^2(x_1+\theta(x_2-x_1),y)}\left \Vert x_1-x_2\right \Vert \\
& \leq C_1\left \Vert \alpha \right \Vert_*\left \Vert x_1-x_2\right \Vert,
\end{align*}
where $C_1=\sup_{(x,y)\in U\times S^{n-1}} \frac{\left \Vert (d_hF)(x,y)\right \Vert}{F^2(x,y)}$.
\end{proof}

\begin{proposition}
\label{condicao de Lipschitz vertical}
Let $U,\widetilde U\subset\mathbb{R}^n$ be open subsets such that $U\subset\subset \widetilde U$. If $F:\widetilde U\times\mathbb{R}^n\rightarrow \mathbb{R}$ is a horizontally $C^1$ family of asymmetric norms, then there exists $C_2>0$ such that
\end{proposition}
\begin{equation*}
\vert F_*(x,\alpha_1)-F_*(x,\alpha_2)\vert \leq C_2\left \Vert \alpha_1-\alpha_2 \right \Vert_*
\end{equation*}
for every $\alpha_1,\alpha_2\in\mathbb{R}^{n*}$ and $x\in U$.
\begin{proof}
Indeed, given $\alpha_1,\alpha_2\in\mathbb{R}^{n*}$ and $x\in U$, we have
\begin{align*}
\vert F_*(x,\alpha_1)-F_*(x,\alpha_2)\vert
&= \left\vert \sup_{y\in\mathbb{R}^n\backslash{0}}\alpha_1\left(\frac{y}{F(x,y)}\right)-\sup_{y\in\mathbb{R}^n\backslash{0}}\alpha_2\left(\frac{y}{F(x,y)}\right)\right\vert \\
& \leq \sup_{y\in\mathbb{R}^n\backslash{0}} \left \Vert \frac{y}{F(x,y)} \right \Vert \left \Vert \alpha_1-\alpha_2 \right \Vert_* 
 = \sup_{y\in S^{n-1}} \left \Vert \frac{y}{F(x,y)} \right \Vert \left \Vert \alpha_1-\alpha_2 \right \Vert_* \\
& \leq C_2 \left \Vert \alpha_1-\alpha_2 \right \Vert_*,
\end{align*}
where $C_2=\sup_{(x,y)\in U\times S^{n-1}}  \frac{1}{F(x,y)}$.
\end{proof}

\begin{theorem}
\label{fam dual local lipschtz}
If $F:\widetilde U\times\mathbb{R}^n\rightarrow \mathbb{R}$ is a horizontally $C^1$ family of asymmetric norms. Then the respective family of dual asymmetric norms $F_*:\widetilde U\times\mathbb{R}^{n*}\rightarrow \mathbb{R}$ is locally Lipschitz.
\end{theorem}
\begin{proof}
It follows from Propositions \ref{condicao de Lipschitz horizontal} and \ref{condicao de Lipschitz vertical} and the fact that every $x\in \widetilde U$ admit a convex neighborhood which is compactly embedded in $\widetilde U$.
\end{proof}

Analogously to Definition \ref{definicaofortementeconvexa}, the next two definitions aim to define the concept of strongly convexity for a horizontally $C^1$ family of asymmetric norms.

\begin{definition}
\label{familia suave de normas euclidianas}
Let $U\subset \mathbb{R}^n$ be an open subset. We say that $\check F :U\times\mathbb{R}^n\rightarrow \mathbb{R}$ is a smooth family of asymmetric norms if $\check F$ is smooth on $U\times \mathbb{R}^n\backslash{0}$ and $\check F(x,\cdot):\mathbb{R}^n\rightarrow \mathbb{R}$ is an asymmetric norm for each $x\in U$.
\end{definition}

\begin{definition}
\label{definicao fortemente convexa com respeito a familia}
Let $\check F :U\times\mathbb{R}^n\rightarrow \mathbb{R}$ be a smooth family of asymmetric norms and $F:U\times\mathbb{R}^n\rightarrow \mathbb{R}$ be a horizontally $C^1$ family of asymmetric norms. We say that $F$ is strongly convex with respect to $\check F$ if
\begin{equation*}
F^2(x,z)\geq F^2(x,y)+\alpha(z-y) +\check F^2(x,z-y)
\end{equation*}
for every $y,z\in \mathbb{R}^n$, $x\in U$ and $\alpha\in\partial F^2(x,y)$.
\end{definition}

We have the following result about the variation of $d_vF_\ast$ along the vertical direction.

\begin{proposition}
\label{condicao de lipschitz vertical para dF*}
Let $U,\widetilde U\subset\mathbb{R}^n$ be open subsets with $U\subset\subset \widetilde U$. 
Let $\check F:\widetilde U\times\mathbb{R}^n\rightarrow \mathbb{R}$ be a smooth family of asymmetric norms and $F:\widetilde U \times\mathbb{R}^n\rightarrow \mathbb{R}$ be a horizontally $C^1$ family of strongly convex asymmetric norms with respect to $\check F$. Then there exists $C_3>0$ such that
\begin{equation}
\label{lipschitz vertical para nabla}
\left \Vert (d_v{F_*}^2)(x,\alpha_1)-(d_v{F_*}^2)(x,\alpha_2) \right \Vert \leq C_3 \left \Vert \alpha_1-\alpha_2 \right \Vert_*
\end{equation}
for every $x\in U$ and $\alpha_1,\alpha_2\in\mathbb{R}^{n*}$.
\end{proposition}
\begin{proof}
Due to $U\subset \subset \widetilde U$, there exist $c_1, c_2 >0$ such that $c_1 \Vert y \Vert \leq \check F(x,y) \leq c_2 \Vert y \Vert$ for every $(x,y) \in U \times \mathbb R^n$. Proceeding analogously to Theorem \ref{diferencialduallipschitz}, we have
\begin{equation}
\label{familiadiferencialduallipschitz}
\left \Vert (d_v{F_*}^2)(x,\alpha_1)-(d_v{F_*}^2)(x,\alpha_2) \right \Vert \leq \frac{2}{(c_2)^2}\left \Vert \alpha_1-\alpha_2 \right \Vert_*,
\end{equation}
for every $x\in U$ and $\alpha_1,\alpha_2\in\mathbb{R}^{n*}$, what settles the proposition.
\end{proof}

The next result shows the continuity of $d_vF_\ast$ and it will be used in the proof of Proposition \ref{condicao Lipschitz horizontal para dF*}.

\begin{proposition}
\label{continuidade de dF*}
Let $\check F :U\times\mathbb{R}^n\rightarrow \mathbb{R}$ be a smooth family of asymmetric norms and $F:U\times\mathbb{R}^n\rightarrow \mathbb{R}$ be a horizontally $C^1$ family of strongly convex asymmetric norms with respect to $\check F$. Then, the application $d_v{F_*}^2$ is continuous.
\end{proposition}
\begin{proof}
Let us to show that given a sequence $(x_n,\alpha_n)\mapsto (x,\alpha)$ in $U\times\mathbb{R}^{n*}$, then $d_v{F_*}^2(x_n,\alpha_n)\mapsto d_v{F_*}^2(x,\alpha)$. 
We have that
\begin{align*}
 F^2\left(x_n,\frac{1}{4}(d_v{F_*}^2)(x,\alpha)\right)
 \geq & F^2\left(x_n,\frac{1}{4}(d_v{F_*}^2)(x_n,\alpha_n)\right)+ \alpha_n\left(\frac{1}{4}(d_v{F_*}^2)(x,\alpha)-\frac{1}{4}(d_v{F_*}^2)(x_n,\alpha_n)\right) \\
 + & \tilde{F}^2\left(x_n, \frac{1}{4}(d_v{F_*}^2)(x,\alpha)-\frac{1}{4}(d_v{F_*}^2)(x_n,\alpha_n)\right)
\end{align*}
due to (\ref{inversasubdiferencial}), which is equivalent to
\begin{align*}
F^2\left(x_n,(d_v{F_*}^2)(x,\alpha)\right)
 \geq & F^2\left(x_n,(d_v{F_*}^2)(x_n,\alpha_n)\right)+ \alpha_n\left(4(d_v{F_*}^2)(x,\alpha)-4(d_v{F_*}^2)(x_n,\alpha_n)\right) \\
 + & \check F^2\left(x_n,(d_v{F_*}^2)(x,\alpha)-(d_v{F_*}^2)(x_n,\alpha_n) \right).
\end{align*}
Thus
\begin{align*}
& \check F^2\left(x_n,(d_v{F_*}^2)(x,\alpha)-(d_v{F_*}^2)(x_n,\alpha_n) \right) \\
\leq & F^2\left(x_n,(d_v{F_*}^2)(x,\alpha)\right)-F^2\left(x_n,(d_v{F_*}^2)(x_n,\alpha_n)\right) - 4\alpha_n((d_v{F_*}^2)(x,\alpha)) +4\alpha_n((d_v{F_*}^2)(x_n,\alpha_n))\nonumber \\
= & F^2\left(x_n,(d_v{F_*}^2)(x,\alpha)\right)-4{F_*}^2(x_n,\alpha_n)-4\alpha_n((d_v{F_*}^2)(x,\alpha)) + 8{F_*}^2(x_n,\alpha_n).
\end{align*}
In the last equality, the second and fourth terms are due to Lemma \ref{igualdadefundamental} and Euler's theorem respectively. Applying the limit, we have
\begin{align*}
& \lim_{n\to\infty}\check F^2\left(x_n,(d_v{F_*}^2)(x,\alpha)-(d_v{F_*}^2)(x_n,\alpha_n)\right) \\
\leq & \lim_{n\to\infty}F^2\left(x_n,(d_v{F_*}^2)(x,\alpha)\right)-\lim_{n\to\infty}4{F_*}^2(x_n,\alpha_n) - \lim_{n\to\infty}4\alpha_n((d_v{F_*}^2)(x,\alpha))+\lim_{n\to\infty}8{F_*}^2(x_n,\alpha_n)\\
= & F^2\left(x,(d_v{F_*}^2)(x,\alpha)\right)-4{F_*}^2(x,\alpha) - 4\alpha((d_v{F_*}^2)(x,\alpha))+8{F_*}^2(x,\alpha)= 0 \nonumber,
\end{align*}
again due to Lemma \ref{igualdadefundamental} and Euler's theorem.
\end{proof}

The analysis of the horizontal variation of $d_v F_\ast$ will be done in two stages. Lemma \ref{diferenca simetrica} provides a technical step while the Proposition \ref{condicao Lipschitz horizontal para dF*} provides sufficient conditions in order to prove Theorem \ref{dF* localmente Lipschitz}.
\begin{lemma}
\label{diferenca simetrica}
Let $U,\widetilde U$ be open subsets of $\mathbb{R}^n$ such that $U\subset\subset \widetilde U$, $\check F:\widetilde U\times\mathbb{R}^n\rightarrow \mathbb{R}$ be a smooth family of asymmetric norms and $F:\widetilde U\times\mathbb{R}^n\rightarrow \mathbb{R}$ be a horizontally $C^1$ family of strongly convex asymmetric norms with respect to $\check F$. 
Then there exists a $C_4 > 0$ such that
\begin{align*}
& C_4 \left \Vert (d_v{F_*}^2)(x_1,\alpha)-(d_v{F_*}^2)(x_2,\alpha) \right \Vert^2 \\
\leq & F^2(x_1,(d_v{F_*}^2)(x_2,\alpha))+F^2(x_2,(d_v{F_*}^2)(x_1,\alpha)) - F^2(x_1,(d_v{F_*}^2)(x_1,\alpha))-F^2(x_2,(d_v{F_*}^2)(x_2,\alpha))
\end{align*}
for every $x_1,x_2\in U$ and $\alpha\in \mathbb{R}^{n*}$.
\end{lemma}
\begin{proof}
Due to the embedding $U\subset\subset \widetilde U$, there exists $c>0$ such that $\check F$ can be replaced by $c\Vert \cdot \Vert$.
As a consequence of (\ref{inversasubdiferencial}) and Definition \ref{definicao fortemente convexa com respeito a familia}, we obtain
\begin{align*}
c^2\left \Vert (d_v{F_*}^2)(x_1,\alpha)-(d_v{F_*}^2)(x_2,\alpha) \right \Vert^2
\leq & F^2(x_1,(d_v{F_*}^2)(x_2,\alpha))-F^2(x_1,(d_v{F_*}^2)(x_1,\alpha)) \\
&-4\alpha((d_v{F_*}^2)(x_2,\alpha)-(d_v{F_*}^2)(x_1,\alpha))
\end{align*}
and
\begin{align*}
c^2\left \Vert (d_v{F_*}^2)(x_1,\alpha)-(d_v{F_*}^2)(x_2,\alpha) \right \Vert^2
\leq & F^2(x_2,(d_v{F_*}^2)(x_1,\alpha))-F^2(x_2,(d_v{F_*}^2)(x_2,\alpha))\\
& - 4\alpha((d_v{F_*}^2)(x_1,\alpha)-(d_v{F_*}^2)(x_2,\alpha)).
\end{align*}
Adding the above equations we settle the lemma.
\end{proof}

\begin{definition}
$F:U\times\mathbb{R}^n\rightarrow \mathbb{R}$ is a horizontally $C^2$ family of asymmetric norms if it is a horizontally $C^1$ family of asymmetric norms and $\frac{\partial^2 F}{\partial x^i\partial x^j}:U\times \mathbb{R}^n \backslash \{ 0 \}\rightarrow\mathbb{R}^n$ is continuous for every $i,j=1,\ldots,n$.
\end{definition}

\begin{proposition}
\label{condicao Lipschitz horizontal para dF*}
Let $U\subset\subset \widetilde U$ be open subsets of $\mathbb{R}^n$, $\check F:\widetilde U\times\mathbb{R}^n\rightarrow \mathbb{R}$ be a smooth family of asymmetric norms and $F:\widetilde U\times\mathbb{R}^n\rightarrow \mathbb{R}$ be a horizontally $C^2$ family of strongly convex asymmetric norms with respect to $\check F$ such that
\begin{equation}
d_hF^2:\widetilde U \times \mathbb{R}^n\rightarrow {\mathbb R}^{n*} 
\end{equation}
is locally Lipschitz.
Then, given $(x,\beta)\in U\times\mathbb{R}^{n*}$, there are 
neighborhoods $U''\subset U$ of $x$, $V''\subset\mathbb{R}^{n*}$ of $\beta$ and a constant $C_5 > 0$ such that 
\begin{equation*}
\left \Vert d_v{F_*}^2(x_1,\alpha)-d_v{F_*}^2(x_2,\alpha) \right \Vert \leq C_5 \left \Vert x_1-x_2 \right \Vert
\end{equation*}
for every $x_1,x_2\in U''$ and $\alpha\in V''$.
\end{proposition}
\begin{proof}
By Lemma \ref{diferenca simetrica}, there exists $C_4 > 0$ such that
\begin{align}
\label{inequacao diferenca simetrica}
C_4\left \Vert (d_v{F_*}^2)(x_1,\alpha)-(d_v{F_*}^2)(x_2,\alpha) \right \Vert^2
\leq & F^2(x_1,(d_v{F_*}^2)(x_2,\alpha))+F^2(x_2,(d_v{F_*}^2)(x_1,\alpha)) \nonumber \\
& - F^2(x_1,(d_v{F_*}^2)(x_1,\alpha))-F^2(x_2,(d_v{F_*}^2)(x_2,\alpha))
\end{align}
for every $x_1,x_2\in U$ and $\alpha\in \mathbb{R}^{n*}$. Using Taylor series of $F^2(\cdot , (d_v F^2_\ast)(x_2, \alpha))$ around $x_1$ we have
\begin{align*}
F^2(x_3,(d_v{F_*}^2)(x_2,\alpha))
= & F^2(x_1,(d_v{F_*}^2)(x_2,\alpha))+(d_hF^2)(x_1,(d_v{F_*}^2)(x_2,\alpha))(x_3-x_1) \\
& + r_{1,2}(x_3-x_1,\alpha).
\end{align*}
where $r_{1,2}(x_3 - x_1, \alpha)$ is the Lagrange remainder. Replacing $x_3$ by $x_2$, we have
\begin{align}
F^2(x_1,(d_v{F_*}^2)(x_2,\alpha))
= & F^2(x_2,(d_v{F_*}^2)(x_2,\alpha))-(d_hF^2)(x_1,(d_v{F_*}^2)(x_2,\alpha))(x_2-x_1) \nonumber \\
& - r_{1,2}(x_2-x_1,\alpha). \label{equacao serie de Taylor 1}
\end{align}
Inverting the roles of $x_1$ and $x_2$, we have
\begin{align}
F^2(x_2,(d_v{F_*}^2)(x_1,\alpha))
 = & F^2(x_1,(d_v{F_*}^2)(x_1,\alpha))-(d_hF^2)(x_2,(d_v{F_*}^2)(x_1,\alpha))(x_1-x_2) \nonumber \\
& - r_{2,1}(x_1-x_2,\alpha). \label{equacao serie de Taylor 2}
\end{align}
Therefore, replacing (\ref{equacao serie de Taylor 1}) and (\ref{equacao serie de Taylor 2}) in  (\ref{inequacao diferenca simetrica}), we have
\begin{align*}
& C_4\left \Vert (d_v{F_*}^2)(x_1,\alpha)-(d_v{F_*}^2)(x_2,\alpha) \right \Vert^2 \\
\leq & \left \Vert(d_hF^2)(x_1,(d_v{F_*}^2)(x_2,\alpha))-(d_hF^2)(x_2,(d_v{F_*}^2)(x_1,\alpha))\right\Vert_* \left \Vert(x_1-x_2)\right\Vert - r
\end{align*}
for every $x_1,x_2\in U$ and $\alpha\in\mathbb{R}^{n*}$, where $r=r_{1,2}(x_2-x_1,\alpha)+r_{2,1}(x_1-x_2,\alpha)$. 

Since $(x,y)\mapsto(d_hF^2)(x,y)$ is locally Lipschitz, given $(x,\beta)\in U\times\mathbb{R}^{n*}$, there are neighborhoods $U'_1\subset\subset U$ of $x$, $U_2^\prime \subset \subset \mathbb{R}^{n}$ of $d_vF_*^2(x,\beta)$ and a constant $K_1>0$ such that
\begin{equation*}
\left\Vert(d_hF^2)(x_1,y_1)-(d_hF^2)(x_2,y_2)\right\Vert_*\leq K_1\left\Vert(x_1-x_2,y_1-y_2)\right\Vert
\end{equation*}
for every $(x_1,y_1), (x_2,y_2)\in U^\prime_1 \times U^\prime_2$. 

We have that ${(d_v{F_*}^2)}^{-1}(U^\prime_2)\cap (U^\prime_1 \times \mathbb{R}^{n*})$ is a neighborhood of $(x,\beta)$ as a consequence of  Proposition \ref{continuidade de dF*}. 
Consider a convex neighborhood $U''\times V''\subset\subset {(d_v{F_*}^2)}^{-1}(U_2^\prime)\cap (U_1^\prime\times \mathbb{R}^{n*})$ of $(x,\beta)$. 
Thus,
\begin{align*}
& \left \Vert (d_v{F_*}^2)(x_1,\alpha)-(d_v{F_*}^2)(x_2,\alpha) \right \Vert^2 \\
\leq & \frac{K_1}{C_4}\left \Vert \left(x_1-x_2,(d_v{F_*}^2)(x_2,\alpha)-(d_v{F_*}^2)(x_1,\alpha)\right)\right\Vert \left \Vert(x_1-x_2)\right\Vert-\frac{r}{C_4} \\
\leq & \frac{2K_1}{C_4}\max\{\left \Vert(x_1-x_2)\right\Vert,\left\Vert (d_v{F_*}^2)(x_2,\alpha)-(d_v{F_*}^2)(x_1,\alpha)\right\Vert\}\left \Vert(x_1-x_2)\right\Vert - \frac{r}{C_4}
\end{align*}
for every $x_1,x_2\in U''$ and $\alpha\in V''$.

If $(x_1,\alpha),(x_2,\alpha)\in U''\times V''$ are such that
\begin{equation*}
\left\Vert (d_v{F_*}^2)(x_2,\alpha)-(d_v{F_*}^2)(x_1,\alpha)\right\Vert\leq \left \Vert(x_1-x_2)\right\Vert,
\end{equation*}
then there is nothing to prove. Suppose that $(x_1,\alpha)$, $(x_2,\alpha)\in U''\times V''$ are such that
\begin{equation*}
\left\Vert (d_v{F_*}^2)(x_2,\alpha)-(d_v{F_*}^2)(x_1,\alpha)\right\Vert > \left \Vert(x_1-x_2)\right\Vert.
\end{equation*}
In this case,
\begin{align*}
C_4\left \Vert (d_v{F_*}^2)(x_1,\alpha)-(d_v{F_*}^2)(x_2,\alpha) \right \Vert^2
\leq 2K_1 \left \Vert (d_v{F_*}^2)(x_1,\alpha)-(d_v{F_*}^2)(x_2,\alpha)\right\Vert \left \Vert(x_1-x_2)\right\Vert-r.
\end{align*}
Writing $t=\left \Vert (d_v{F_*}^2)(x_1,\alpha)-(d_v{F_*}^2)(x_2,\alpha) \right \Vert$ and $\left \Vert x_1-x_2 \right \Vert=\epsilon$, it follows that
\[
C_4 t^2-2 K_1 \epsilon t+r\leq 0.
\] 
In particular, the polynomial equation  $C_4 t^2-2K_1\epsilon t+r=0$ admits at least one real root. If $r\geq 0$, then
$$C_4 t^2-2K_1\epsilon t\leq 0$$
and
$$\left\Vert(d_v{F_*}^2)(x_1,\alpha)-(d_v{F_*}^2)(x_2,\alpha)\right\Vert\leq\frac{2K_1}{C_4}\left\Vert(x_1-x_2)\right\Vert$$
what settles this case. Finally suppose that $r<0$. In this case $t$ is bounded above by the largest root of the quadratic equation $C_4 t^2-2K_1\epsilon t+r=0$, which is
$$\frac{K_1\epsilon+\sqrt{K_1^2\epsilon^2- C_4 r}}{C_4}.$$
Using the Lagrange remainder in the Taylor series, we can write
\begin{equation*}
r_{1,2}(x_2-x_1,\alpha)=\frac{1}{2}(d_h^2F^2)(x_1+\theta_{1,2}(x_2-x_1),(d_v{F_*}^2)(x_2,\alpha))(x_2-x_1)^2
\end{equation*}
and
\begin{equation*}
r_{2,1}(x_1-x_2,\alpha)=\frac{1}{2}(d_h^2F^2)(x_2+\theta_{2,1}(x_1-x_2),(d_v{F_*}^2)(x_1,\alpha))(x_1-x_2)^2,
\end{equation*}
where $\theta_{1,2},\theta_{2,1}\in[0,1]$. Set
\begin{align*}
K_2=\inf_{\substack{\tilde x_1 \in U'' \\ \tilde x_2 \in U_2^\prime \\ \left \Vert v \right \Vert=1}} & ((d_h^2F^2)(\tilde x_1,\tilde x_2))v^2.
\end{align*}
Thus $K_2<0$ because $r<0$, the inequality $-r\leq-K_2\epsilon^2$ holds and 
\begin{equation*}
t\leq\frac{K_1+\sqrt{K_1^2-C_4 K_2}}{C_4}\epsilon.
\end{equation*}
Therefore,
\begin{equation*}
\left \Vert (d_v{F_*}^2)(x_1,\alpha)-(d_v{F_*}^2)(x_2,\alpha) \right \Vert\leq C_5 \left \Vert x_1-x_2 \right \Vert,
\end{equation*}
where $C_5 = \frac{K_1 +\sqrt{K_1^2-C_4 K_2}}{C_4}$, what settles the proposition.
\end{proof}

\begin{theorem}
\label{dF* localmente Lipschitz}
Under the hypotheses of Proposition \ref{condicao Lipschitz horizontal para dF*},
$$d_v{F_*}^2:U\times\mathbb{R}^{n*}\rightarrow \mathbb{R}^n$$
is locally Lipschitz.
\end{theorem}
\begin{proof}
The proof follows from Propositions \ref{condicao de lipschitz vertical para dF*} and \ref{condicao Lipschitz horizontal para dF*}.
\end{proof}

\section{A Locally Lipschitz Case}
\label{A Locally Lipschitz Case}
In this section we will use the concepts introduced in the previous section in order to present a family of of Pontryagin type $C^0$-Finsler structures whose extended geodesic field is a locally Lipschitz vector field.

Let $(M,F)$ be a $C^0$-Finsler manifold.
If $\tau: TU \rightarrow U \times \mathbb R^n$ is a local trivialization of an open subset $U$ of $M$ and $\phi=(x^1,\ldots,x^n): U \rightarrow \mathbb R^n$ is a coordinate system, then $\tau_\phi := (\phi \times \text{id}) \circ \tau$ is a coordinate system on $TU$.
Suppose that for each $p \in M$ there exists a coordinate system $\tau_\phi$ of $TU$ such that $F \circ (\tau_\phi)^{-1} (x^1, \ldots, x^n, v^1, \ldots, v^n)$ is a horizontally $C^1$ family of asymmetric norms.
It is straightforward that the horizontal smoothness at $p$ depends only on $\tau$ (it doesn't depend on the choice of $\phi$).

\begin{proposition}
Let $(M,F)$ be a $C^0$-Finsler manifold and suppose that $F$ is locally represented in a coordinate system $\tau_\phi$ as a horizontally $C^1$ family of asymmetric norms. Then $F$ is of Pontryagin type.
\end{proposition}
\begin{proof}
Considering the control set $C$ as the unit sphere $S^{n-1}$, for each $u\in C$ corresponds a smooth vector field $\check X_u$ given by
\[
(x^1, \ldots, x^n) \mapsto (x^1, \ldots x^n, u^1, \ldots, u^n)
\]
with respect to the coordinates $\phi$ and $\tau_\phi$ on $U$ and $TU$ respectively. Define the family of vector fields
\[
X_u = \frac{\check X_u}{F(\check X_u)}.
\]
It is straightforward that $\{ X_u; u\in S^{n-1}\}$, satisfy all the conditions of Definition \ref{horizontally smooth} with respect to $\tau_\phi$ instead of $\phi_{TU}$. 
But Remark \ref{mudanca de coordenadas - trivializacao local} states that this is enough to prove that $X_u$, $u\in C$, satisfy all the conditions of Definition \ref{horizontally smooth} with respect to $\phi_{TU}$.
Therefore $(M,F)$ is a Pontryagin type $C^0$-Finsler manifold.
\end{proof}

\begin{definition}
\label{fortemente convexa variedades}
Let $(M,F)$ be a $C^0$-Finsler manifold and let $\check F$ be a Finsler structure on $M$. We say that $F$ is strongly convex with respect to $\check F$ if 
\begin{equation}
F^2(x,z)\geq F^2(x,y)+ \alpha(z-y) +\check F^2(x,z-y)
\end{equation}
for every $x\in M$, $y,z\in T_xM$ and $\alpha\in\partial F^2(x,y)$.
\end{definition}

The next lemma state that the concept of strong convexity can be transferred to $\tau_\phi(TU)$.
\begin{lemma}
Let $(M,F)$ be a $C^0$-Finsler manifold and $U$ be an open subset of $M$.
Let $\tau_\phi:TU \rightarrow \phi(U) \times \mathbb{R}^n$ be a parameterization of $TU$ such that $F_{\tau_\phi} = F\circ (\tau_\phi)^{-1}$ is a horizontally $C^1$ family of asymmetric norms.
Then $F$ is strongly convex with respect to a Finsler structure $\check F$ iff $F_{\tau_\phi}$ is strongly convex with respect to $\check F_{\tau_\phi}:=\check F \circ (\tau_\phi)^{-1}$. 
Moreover $\alpha(x, \cdot)\in\partial F^2(x,y)$ iff $\alpha_{\tau_\phi}(x,\cdot)\in\partial F^2_{\tau_\phi}(\tau_\phi(x,y))$.
\end{lemma}
\begin{proof}
It is enough to notice that

\[
F^2(x,z) \geq F^2(x,y) + \alpha (x,z-y) + \check F(x,z-y)
\]
if and only if
\[
F_{\tau_\phi}^2(\tau_\phi (x,z)) \geq  F_{\tau_\phi}^2(\tau_\phi (x,y)) + \alpha_{\tau_\phi} (\tau_\phi(x,y-z)) + \check F^2_{\tau_\phi}(\tau_\phi(x,y-z)).
\]
\end{proof}

\begin{theorem}
\label{campo geodesico Lipschitz}
Let $(M,F)$ be a $C^0$-Finsler manifold which is strongly convex with respect to a Finsler structure $\check F$ and consider an open subset $U \subset M$. 
Suppose that there exist a local trivialization $\tau: TU \rightarrow U \times \mathbb R^n$ and a coordinate system $\phi:U \rightarrow \mathbb R^n$ such that $F_{\tau_\phi} = F\circ (\tau_\phi)^{-1}$ is a horizontally $C^2$ family of asymmetric norms and that $d_hF^2_{\tau_\phi}$ is locally Lipschitz. Then the extended geodesic field $\mathcal E$ of $(U,F)$ is a locally Lipschitz vector field on $T^\ast U$. 
\end{theorem}

\begin{proof}
Consider the vector field
\begin{equation}
\label{hamiltoniano}
{\vec { H}}_u = f_i(x,u) \frac{\partial}{\partial x_i} - \alpha_j \frac{\partial f_j}{\partial x_i} \frac{\partial}{\partial \alpha_i}
\end{equation}

which, in coordinates, is given by the application

\begin{equation}
\label{coordhamiltoniano}
(x,\alpha,u) \mapsto \left(x_1,\ldots,x_n,\alpha_1,\ldots,\alpha_n,f_1(x,u),\ldots,f_n(x,u), \alpha_j \frac{\partial f_j}{\partial x_1}(x,u),\ldots, \alpha_j \frac{\partial f_j}{\partial x_n}(x,u)\right).
\end{equation}
Notice that (\ref{hamiltoniano}) is locally Lipschitz if and only if $f_i(x,u)$ and $\frac{\partial f_j}{\partial x_i}(x,u)$ are locally Lipschitz for $i,j=1,\ldots,n$ or, 
equivalently, if and only if $(x,u)\mapsto X_u(x)$ and $(x,u)\mapsto \frac{\partial X_u}{\partial x_i}(x)$ are locally Lipschitz. These conditions hold due to the hypotheses of this theorem.

The extended geodesic field $\mathcal E(x,\alpha)$ is defined associating to each $(x,\alpha)\in T^*M$ the vectors $\vec {H}_u(x,\alpha)$, where $u(x,\alpha)\in S^{n-1}$ are such that $X_{u(x,\alpha)}(x)$ maximizes $\alpha$. In the case of strictly convex asymmetric norms, this correspondence is unique and $\mathcal E(x, \alpha)$ becomes a vector field. In addition, we are under the conditions of Theorem \ref{dF* localmente Lipschitz}, what implies that $(x,\alpha)\mapsto u(x,\alpha)$ given by $u(x,\alpha)=\Pi\left(\frac{(d_v{F_*}^2)(x,\alpha)}{F(x,{(d_v{F_*}^2)(x,\alpha)})}\right)$ (where $\Pi$ is the radial projection on $S^{n-1}$) is locally Lipschitz. Therefore it follows that

\begin{align*}
(x,\alpha)\mapsto & \left( \frac{}{} x_1,\ldots,x_n,\alpha_1,\ldots,\alpha_n,f_1(x,u(x,\alpha)),\ldots,f_n(x,u(x,\alpha)),\right. \\
& \left. \alpha_j \frac{\partial f_j}{\partial x_1}(x,u(x,\alpha)),\ldots, \alpha_j \frac{\partial f_j}{\partial x_n}(x,u(x,\alpha)) \right)
\end{align*}
is locally Lipschitz, that is, $(x,\alpha)\mapsto \mathcal{E}(x,\alpha)$ is locally Lipschitz.
\end{proof}

\section{Extended geodesic field on homogeneous spaces}
\label{Extended geodesic field on homogeneous spaces}

In this section we will show that homogeneous spaces endowed with $G$-invariant $C^0$-Finsler structures are Pontryagin type $C^0$-Finsler manifolds. Moreover, if the $C^0$-Finsler structure satisfy the strong convexity condition (Definition \ref{definicaofortementeconvexa}) in some tangent space, then $\mathcal E$ is a locally Lipschitz vector field. This section is independent of Sections \ref{Horizontally C1 family of asymmetric norms subsection} and \ref{A Locally Lipschitz Case}.

Let $G$ be a Lie group and $H \subset G$ be a closed subgroup. 
Let $\mathfrak{g}$ and $\mathfrak{h}$ be the lie algebras of $G$ and $H$ respectively and $\mathfrak m$ be a vector subspace of $\mathfrak{g}$ such that $\mathfrak{g}=\mathfrak m\oplus\mathfrak{h}$. 
Let $a:G\times G/H\to G/H$ be the canonical left action, $a_g:=a(g,\cdot) :G/H\to G/H$ and $\pi:G \rightarrow G/H$ be the projection map. 
As usual, we denote the action of $g$ on $g^\prime H \in G/H$ by $gg^\prime H$. 
Consider a neighborhood $U$ of $0\in\mathfrak m$ such that the collection of applications 
\begin{equation}
\label{atlas G/H}
\{\sigma_g: U \to G/H;\sigma_g(x)=g\pi(\exp(x))\}_{g\in G} 
\end{equation}
is an atlas of $G/H$.

Suppose that $(G/H,F)$ is a $G$-invariant $C^0$-Finsler manifold. 
Given $gH\in G/H$, we construct a family of unit vector fields on $g\exp (U)H$ as follows: Consider the unit sphere $S_F[eH,0,1]$ on $(T_{eH}G/H,F)$ and 
$$d((\pi\circ \exp)^{-1})_{eH}(S_F(eH))=S\subset\mathfrak m.$$
Define the family of vector fields 
\begin{align*}
X:& U\times S \to TG/H \\
& (x,u)\mapsto X_u(x)
\end{align*}
by $X_u(x)=d(a_{g\exp(x)})_{eH}(d(\pi\circ \exp)_0(u))$. Since $F$ is $G$-invariant, it follows that $X_u(x)$ is a unit vector for every $(x,u)$.
For the sake of convenience we choose $u \in S$ instead of $u \in S^{n-1}$, but we can identify $S$ to $S^{n-1}$ through a radial Lipschitz map. 

Let's check that $X_u(x)$ satisfies the conditions of Definition \ref{horizontally smooth}.
\begin{enumerate}
\item Since $d(a_{g\exp(x)})_{eH}$ and $d(\pi\circ \exp)_0$ are isomorphisms, it follows that

$$u\mapsto X_u(x)$$
is a homeomorphism from $S$ onto $S_F[g\exp(x)H,0,1]$, for every $x\in U$.
\item Note that the application $(x,u)\mapsto X_u(x)$ gives a natural $C^\infty$ extension from $U\times\mathfrak m$ onto $TG/H$ given by
\begin{align}
\label{extensao de X}
\bar{X}: & U\times \mathfrak m\longrightarrow TG/H \nonumber \\
& (x,u)\mapsto \bar{X}_u(x)=d(a_{g\exp(x)})_{eH}(d(\pi\circ \exp)_0(u))
\end{align}
Therefore, we can conclude that $(x,u)\mapsto X_u(x)$ is continuous.
\item By (\ref{extensao de X}) we can also conclude that
$$(x,u)\mapsto \frac{\partial X_u(x)}{\partial x^i}$$
are continuous with respect to a linear coordinate $(x^1, \ldots, x^n)$ on $U$. 
\end{enumerate}
Therefore the $C^0$-Finsler manifold $(G/H,F)$ is of Pontryagin type.
As a consequence of (\ref{extensao de X}), we also have that
\[
(x,u)\mapsto \vec H_u = f_i(x,u) \frac{\partial}{\partial x_i} -   \alpha_j \frac{\partial f_j}{\partial x_i}(x,u) \frac{\partial}{\partial \alpha_i} 
\]
is locally Lipschitz.

Now suppose there is an asymmetric norm $\check F$ on $\mathfrak m$ such that the asymmetric norm with unit sphere $S$ is strongly convex with respect to $\check F$.
We claim that $(x,\alpha) \mapsto u(x,\alpha)$ is locally Lipschitz, what is enough to prove that the extended geodesic field is locally Lipschitz.

For each $x\in U$, consider the transpose linear operator 
\[
d(a_{g\exp(x)}\circ \pi\circ \exp)^*_{0}:T^*_{g\exp(x)H}G/H\to \mathfrak m^*
\]
and a functional $\alpha_{g\exp(x)H}\in T^*_{g\exp(x)H}G/H$. Due to the equality
\begin{align*}
& \alpha_{g\exp(x)H}(d(a_{g\exp(x)})_{eH}(d(\pi\circ \exp)_0(u))) \\
 = & (d(a_{g\exp(x)}\circ \pi\circ \exp)^*_{0}(\alpha_{g\exp(x)H}))(u), 
\end{align*}
the vector $X_u(x)\in T_{g\exp(x)H}G/H$ maximizes the functional $\alpha_{g\exp(x)H}$ if and only if $u\in S$ maximizes the functional $d(a_{g\exp(x)}\circ \pi\circ \exp)^*_{0}(\alpha_{g\exp(x)H})\in \mathfrak m^*$.
Thus the application $\alpha_{g\exp(x)H}\mapsto u(\alpha_{g\exp(x)H})$, where $u(\alpha_{g\exp(x)H})$ is such that $X_{u(\alpha_{g\exp(x)H})}(x)$ maximizes $\alpha_{g\exp(x)H}$, is given by
\begin{equation}
\label{xi u(xi)}
\alpha_{g\exp(x)H}\mapsto d(a_{g\exp(x)}\circ \pi\circ \exp)^*_{0}(\alpha_{g\exp(x)H})\mapsto u(\alpha_{g\exp(x)H}),
\end{equation}
where $u(\alpha_{g\exp(x)H})$ maximizes $d(a_{g\exp(x)}\circ \pi\circ \exp)^*_{0}(\alpha_{g\exp(x)H})$. Notice that the application $d(a_{g\exp(x)}\circ \pi\circ \exp)^*_{0}(\alpha_{g\exp(x)H})\mapsto u(\alpha_{g\exp(x)H})$ is locally Lipschitz as a consequence of Theorem \ref{diferencialduallipschitz}. Therefore (\ref{xi u(xi)}) is locally Lipschitz and the extended geodesic field given by
\begin{equation*}
(x,\alpha_{g\exp(x)})\mapsto \vec{H}_{u(\alpha_{g\exp(x)})}
\end{equation*}
is also locally Lipschitz. This proves the following theorem.

\begin{theorem}
\label{localmente lipschitz homogeneo}
Let $G$ be a Lie group and $H$ be a closed subgroup of $G$. If $(G/H,F)$ is a $C^0$-Finsler manifold and $F$ is $G$-invariant, then $F$ is of Pontryagin type. Moreover, if $F$ restricted to some tangent space is strongly convex, then the extended geodesic field defined locally by
$$(x,\alpha_{g\exp(x)})\mapsto \vec{H}_{u(\alpha_{g\exp(x)})}$$
is a locally Lipschitz vector field.
\end{theorem}

\section{Examples}
\label{Exemplos}

A quasi-hyperbolic plane is the Lie group $G=\mathbb R^2_+ = \{(x^1, x^2)\in \mathbb R^2; x^2>0\}$ with identity element $(0,1)$ and product $(x^1,x^2)\cdot (z^1, z^2) = (x^2 z^1 + x^1,x^2 z^2)$ endowed with a left invariant symmetric $C^0$-Finsler structure $F$, that is, $F(x, \cdot)$ is a norm for every $x$. Notice that the left multiplication by $(x^1,x^2)$ is equivalent to a homothety by $x^2$ followed by a horizontal translation by $x^1$.
From Theorem \ref{localmente lipschitz homogeneo}, it follows that quasi-hyperbolic planes are Pontryagin type $C^0$-Finsler manifolds.
In \cite{Gribanova}, Gribanova proves that the left invariant symmetric $C^0$-Finsler structures on $G$ are given by 
\begin{equation}
\label{caracteriza plano qh}
F(x^1,x^2,y^1,y^2) = \frac{F(0,1,y^1, y^2)}{x^2}:=\frac{F_{e}(y^1,y^2)}{x^2},
\end{equation}
where $F_e$ is an arbitrary norm on the Lie algebra $\mathfrak{g}$.
In addition, she classifies all minimizing paths of $(G,F)$.
She uses the symmetries of $G$ and the PMP in order to calculate the minimizing paths explicitly.

In this chapter, we study particular examples of quasi-hyperbolic planes according to the concept of extended geodesic field. 
In the first example, we consider a non-strictly convex structure and in the second example a strongly convex one. 
In the first example we analyze how $\mathcal E$ allow us to calculate geodesics even when $\mathcal E$ isn't a vector field.
In the second example, $\mathcal E$ is a locally Lipschitz vector field and the local existence and uniqueness of integral curves of $\mathcal E$ holds.
Although it seems to be a similar to the Finsler case, we show that this kind of geodesic structure can't be represented by a locally Lipschitz vector field on $TM$.

Let us present some calculations that hold for every quasi-hyperbolic plane. 
The unit vector fields in terms of the control $(u^1,u^2)$ are given by
\begin{eqnarray}
f^1(x,u)\frac{\partial}{\partial x^1} + f^2(x,u)\frac{\partial}{\partial x^2} & = & \frac{u^1 x^2}{F_e(u^1,u^2)}\frac{\partial}{\partial x^1} + \frac{u^2 x^2}{F_e(u^1,u^2)}\frac{\partial}{\partial x^2}. \label{unit vector field}
\end{eqnarray}
The Hamiltonian is given by
\begin{eqnarray}
H_u(x,\alpha) = \frac{x^2}{F_e(u^1, u^2)}\left( \alpha_1 u^1 + \alpha_2 u^2 \right), \label{hamiltonian constant}
\end{eqnarray}
which is a constant $C_0>0$ along the minimizing paths.
Denote by $u(x,\alpha)$ the subset of controls that maximizes $H_u(x,\alpha)$. 
Notice that the vector
\[
\frac{u^1(x,\alpha) x^2}{F_e(u(x,\alpha))}\frac{\partial}{\partial x^1} + \frac{u^2(x,\alpha) x^2}{F_e(u(x,\alpha))}\frac{\partial}{\partial x^2} \in S_F[x,0,1]
\]  
that maximizes $H_u(x,\alpha)$ is proportional to $u$ and it doesn't depend on $x$. 
Therefore we can write $u(\alpha)=u(x,\alpha)$.
This remark will help us to do the analysis of the examples of this chapter.

The Hamiltonian vector field is given by
\begin{equation}
\label{equacao Hamilton geral}
\vec H_u(x,\alpha) = \left\{ 
\begin{array}{rcl}
\frac{dx^1}{dt} & = & \frac{x^2 u^1}{F_e(u)}; \\
\frac{dx^2}{dt} & = & \frac{x^2 u^2}{F_e(u)}; \\
\frac{d\alpha_1}{dt} & = & 0; \\
\frac{d\alpha_2}{dt} & = & - \frac{\alpha_i u^i}{F_e(u)},
\end{array}
\right.
\end{equation}
the extended geodesic field is given by
\begin{equation}
\label{equacao Hamilton estendido}
\mathcal E(x,\alpha) = \left\{ 
\begin{array}{rcl}
\frac{dx^1}{dt} & = & \frac{x^2 u^1(\alpha)}{F_e(u(\alpha))}; \\
\frac{dx^2}{dt} & = & \frac{x^2 u^2(\alpha)}{F_e(u(\alpha))}; \\
\frac{d\alpha_1}{dt} & = & 0; \\
\frac{d\alpha_2}{dt} & = & - \frac{\alpha_i u^i(\alpha)}{F_e(u (\alpha))},
\end{array}
\right.
\end{equation}
where (\ref{equacao Hamilton geral}) and (\ref{equacao Hamilton estendido}) are due to (\ref{campohamiltonianocoordenadas2}), (\ref{unit vector field}), (\ref{hamiltonian constant}) and the definition of $\mathcal E$.
The key feature here is that if $(u(t),x(t),\alpha(t))$ is an integral curve of $\mathcal E$, then
\begin{equation}
\label{equacao Hamilton}
\left\{ 
\begin{array}{rcl}
\frac{dx^1}{dt} & = & \frac{x^2(t) u^1(t)}{F_e(u(t))}; \\
\frac{dx^2}{dt} & = & \frac{x^2(t) u^2(t)}{F_e(u(t))}; \\
\frac{d\alpha_1}{dt} & = & 0; \\
\frac{d\alpha_2}{dt} & = & - \frac{C_0}{x^2(t)}
\end{array}
\right.
\end{equation}
holds, where the last equation is due to the constancy of (\ref{hamiltonian constant}) along the curve. 
Here $u(t) = u(\alpha(t))$ must maximize $H_u(x(t),\alpha(t))$ in order to have the PMP satisfied.

Now we go to the analysis of each case separately.

\subsection{A non-strictly convex case}
\label{A non-strictly convex case}

In this example suppose that $F_e$ in (\ref{caracteriza plano qh}) is a norm such that its unit sphere is a regular hexagon $H$ with vertices at 
\[
\{ \pm (0,1), \pm (\sqrt{3}/2,1/2), \pm (\sqrt{3}/2, -1/2)\}.
\]
The direction of these vertices are called preferred directions (see \cite{Fukuoka-large-family}).
If $(\alpha_1, \alpha_2)$ is a positive multiple of $(1,0)$, then $(u^1,u^2) \in u(\alpha)$ iff the Euclidean angle between $(u^1,u^2)$ and $(1,0)$ is in the interval $[-\pi/6, \pi/6]$.
The same type of analysis holds when $(\alpha_1, \alpha_2)$ is a positive multiple of $\{(-1,0), \pm (1/2, \sqrt{3}/2),$ $\pm (-1/2,\sqrt{3}/2)\}$.
On the other hand if the Euclidean angle between $(1,0)$ and $(\alpha_1,\alpha_2)$ is in the interval $(0, \pi/3)$, then $u(\alpha) = (\sqrt{3}/2, 1/2)$, which is a vertex of $H$.
More in general, if the Euclidean angle between $(1,0)$ and $(\alpha_1, \alpha_2)$ is in the interval $(k\pi/3,(k+1)\pi/3)$, $k\in \{1,2,3,4,5\}$, then $u(\alpha)$ is the vector that makes an Euclidean angle $k\pi /3+\pi / 6$ with $(1,0)$, which are also vertices of $H$. The expression (\ref{equacao Hamilton estendido}) states that $\mathcal E(x,\alpha)$ isn't a vector field on $T^\ast M$ because $\mathcal E(x,\alpha)$ can have more that one element for a fixed $(x,\alpha)$.

Let us calculate the solutions of (\ref{equacao Hamilton}) for this particular $F_e$.
The function $\alpha_1$ is constant and there are two cases to consider:

\

First case: $\alpha_1(t) \equiv 0$.

\

Suppose that $\alpha_2(0) > 0$.
We have that $C_0=x^2(0).\alpha_2(0)$.
It is not difficult to see that the unique minimizing path is the projection $(x^1(t),x^2(t))$ of the solution $(x^1(t),x^2(t), \alpha_1(t), \alpha_2(t)) = (x^1(0),x^2(0)e^t, 0, \alpha_2(0)e^{-t})$, where $t \in \mathbb R$.

The analysis for $\alpha_2(0)<0$ is analogous.
The unique minimizing geodesic satisfying the initial conditions $(x^1(0),x^2(0),\alpha_1(0),\alpha_2(0))$ is given by $(x^1(0),$ $x^2(0)e^{-t})$, $t\in \mathbb R$.

In both cases the minimizing path is a vertical line in $G$.

\

Second case: $\alpha_1(t) \equiv \alpha_1(0) \neq 0$.

\

In order to fix ideas, suppose that $\alpha_1(t)  > 0$.

\

First subcase:  $u(\alpha(0))=(0,1)$.

\

The analysis follows in the same way as in the first case.
The solution is given by $(x^1(t),x^2(t),\alpha_1(t),$ $\alpha_2(t))=(x^1(0),x^2(0)e^t, \alpha_1(0),\alpha_2(0)e^{-t})$ and this solution holds in the maximum interval 
\[
t \in I_1 := \left(-\infty, \ln \left( \frac{\alpha_2(0)}{ \sqrt{3}.\alpha_1(0)} \right) \right).
\]
At $t_1 = \ln (\frac{\alpha_2(0)}{ \sqrt{3}.\alpha_1(0)})$, we have that $\alpha(t_1)=(\alpha_1(0),\sqrt{3}.\alpha_1(0))$ and $u(\alpha(t_1))$ is given by the vectors $u$ such that the Euclidean angle between $(1,0)$ and $u$ is in the interval $[\pi/6,\pi/2]$.
A little bit after this point, we have that $u(\alpha(t)) = (\sqrt{3}/2, 1/2)$ due to the fourth equation of (\ref{equacao Hamilton}).
The closure of the trace of $(x^1(t),x^2(t))$ for $t\in I_1$ has Euclidean length 
\[
\frac{C_0}{\sqrt{3}.\alpha_1(0)} = \frac{ x^2(0)\alpha_2(0)}{\sqrt{3}.\alpha_1(0)}.
\]

\

Second subcase: $u(\alpha(0)) = (\sqrt{3}/2, 1/2)$.

\

We have that
\[
C_0 = x^2(0)\left( \alpha_1(0)\frac{\sqrt{3}}{2}  + \alpha_2(0)\frac{1}{2}\right).
\]
If the initial conditions $(x^1(0),x^2(0), \alpha_1(0),\alpha_2(0))$ are given, then
\begin{eqnarray}
\label{hamilton2}
\left\{
\begin{array}{lll}
x^1(t) & = & \left( x^1(0) - \sqrt{3}x^2(0) \right) + \sqrt{3}x^2(0) e^{t/2}; \\
x^2(t) & = & x^2(0)e^{t/2}; \\
\alpha_1(t) & = & \alpha_1(0); \\
\alpha_2(t) & = & (\sqrt{3}.\alpha_1(0) + \alpha_2(0))e^{-t/2}-  \sqrt{3}.\alpha_1(0).
\end{array}
\right.
\end{eqnarray}
Let $t_2$ and $t_3$ such that $\alpha(t_2)=(\alpha_1(0),\sqrt{3}.\alpha_1(0))$ and $\alpha(t_3)=(\alpha_1(0),0)$, which are the boundary of the maximum interval $(t_2,t_3)$ such that $u(\alpha(t))=(\sqrt{3}/2,1/2)$.
From the fourth equation of (\ref{hamilton2}), we get 
\begin{equation}
\label{t2}
e^{\frac{t_2}{2}} = \frac{\sqrt{3}.\alpha_1(0) + \alpha_2(0)}{2 \sqrt{3}.\alpha_1(0)} = \frac{C_0}{\sqrt{3}.x^2(0).\alpha_1(0)}
\end{equation}
and
\begin{equation}
\label{t3}
e^{\frac{t_3}{2}} = \frac{\sqrt{3}.\alpha_1(0) + \alpha_2(0)}{\sqrt{3}.\alpha_1(0)} = \frac{2.C_0}{\sqrt{3}.x^2(0).\alpha_1(0)},
\end{equation}
what implies
\[
x^2(t_2)=\frac{C_0}{\sqrt{3}.\alpha_1(0)}
\]
and
\[
x^2(t_3) = \frac{2C_0}{\sqrt{3}.\alpha_1(0)}.
\]
From (\ref{t2}) and (\ref{t3}) we can see that $t_2$ and $t_3$ are finite.
Therefore the trace of $(x^1(t),x^2(t))$ restricted to $[t_2,t_3]$ is a line segment parallel to $(\sqrt{3}/2,1/2)$ and with Euclidean length $2C_0/(\sqrt{3}.\alpha_1(0))$. 
When we go a little bit before $t_2$, we have that $u(\alpha(t))=(0,1)$ and if we go a little bit after $t_3$, we have that $u(\alpha(t))=(\sqrt{3}/2, -1/2)$.

\

Third subcase: $u(\alpha(0)) = (\sqrt{3}/2,-1/2)$.

\

This subcase is analogous to the subcase $u(\alpha(0)) = (\sqrt{3}/2, 1/2)$.
If $t_4$ and $t_5$ are the boundary of the maximum interval such that $u(\alpha(t)) = (\sqrt{3}/2,-1/2)$, then the trace of $(x^1(t),x^2(t))$ restricted to $[t_4,t_5]$ is a line segment parallel to $(\sqrt{3}/2, -1/2)$ with Euclidean length $2C_0/(\sqrt{3}.\alpha_1(0))$ such that the $x^2$ coordinate of its lowest point is given by $x^2(t_5) = C_0/(\sqrt{3}.\alpha_1(0))$ and the $x^2$ coordinate of its highest point is given by $x^2(t^4) = 2C_0/(\sqrt{3}.\alpha_1(0))$. 
When we go a little bit before $t_4$, we have that $u(\alpha(t))=(\sqrt{3}/2,1/2)$ and if we go a little bit after $t_5$, we have that $u(\alpha(t))=(0, -1)$.

\

Fourth subcase: $u(\alpha(0)) = (0,-1)$.

\

This subcase is analogous to the subcase
$u(\alpha(0)) = (0,1)$.
The projection of the solution of (\ref{equacao Hamilton}) on $G$ is given by $(x^1(t),x^2(t))=(x^1(0), x^2(0)e^{-t})$, it is defined on the maximal interval $(\ln(- \frac{ \sqrt{3}.\alpha_1(0)}{\alpha_2(0)}), \infty)$ and the closure of its trace is the vertical line connecting $(x^1(0),0)$ and $(x^1(0),C_0/(\sqrt{3}.\alpha_1(0))$.
If we go a little bit before $\ln(-\alpha_2(0)/(\sqrt{3}.\alpha_1(0)))$, then $u(\alpha(t))=(\sqrt{3}/2,-1/2)$. 

\

Fifth subcase: $\alpha(0) = (1,0)$. 

\

For $t$ a little bit before $t=0$, we have that $u(\alpha(t))=(\sqrt{3}/2,1/2)$.
For $t$ a little bit after $t=0$, we have that $u(\alpha(t)) = (\sqrt{3/2},-1/2)$.
Therefore $(x^1(0),x^2(0))$ connects a segment of the second subcase with a segment of the third subcase.  

\

Sixth subcase: $\alpha(0) \in \{(1/2, \sqrt{3}/2), (1/2,-\sqrt{3}/2)\}$.

\

It is analogous to the fifth subcase.
The point $(x^1(0),x^2(0))$ connects two sides of the former subcases.

\

The analysis for $\alpha_1(0)<0$ is analogous.

\

From the analysis made before, the solutions above can be connected and they are part of the maximal solution $(x^1(t),x^2(t),\alpha_1(t),\alpha_2(t))$ of (\ref{equacao Hamilton}) defined for every $t\in \mathbb R$. 
The closure of the trace of $(x^1(t),x^2(t))$ is the intersection of $G$ with a regular hexagon with length $2C_0/(\sqrt{3}.\alpha_1(0))$, centered at some point in the $x$-axis and with its sides parallel to the vectors in $\{ (0,1),$ $(\sqrt{3}/2,1/2),$ $(\sqrt{3}/2, -1/2)\}$.
Observe that $C_0$ is constant along the solution.

We leave for the reader the following exercise: If we have two points in $G$ that aren't in the same vertical, then there exists a unique half hexagon (as calculated above) that connects these two points. 

Quasi-hyperbolic planes are complete locally compact length spaces because they are symmetric $C^0$-Finsler manifolds, what implies that every two points can be connected by a minimizing path (see \cite{Burago}). 
Therefore these half hexagons are minimizing paths.
 
\begin{remark}
\label{conclusoes 1}
Here we can draw some conclusions:
\begin{enumerate}
\item Given $(x_0,y_0) \in TG$, if $y_0$ isn't a preferred direction, then there aren't any geodesic $\gamma:(-\epsilon,\epsilon) \rightarrow G$ such that $\gamma(0)=x_0$ and $\gamma^\prime(0) = y_0$.
If $y_0$ is a preferred direction, then there are infinitely many minimizing paths $\gamma:\mathbb R \rightarrow G$ such that $\gamma(0) = x_0$ and $\gamma^\prime(0) = y_0$.
Therefore if $(\gamma(t),\gamma^\prime(t))$ is an absolutely continuous path in $TM$ such that $\gamma(t)$ is a minimizing path, then the trace of $\gamma$ is a subset of a straight line.
This makes any representation of geodesics on $TM$ much more complicated than $\mathcal E$;

\item Every minimizing path is a preferred path, that is, their derivatives are preferred directions almost everywhere. 
If $x_0 \in G$, then the choice of $\alpha_0 \in T_{x_0}^\ast M$ determines the preferred  direction that the geodesic will follow and the exact moment when its direction ``switches'';

\end{enumerate}
\end{remark}

\subsection{A strongly convex example}
\label{A strongly convex example}
\noindent Now we analyze an example where the Finsler structure is strongly convex.

Let us define a norm $F$ on $\mathfrak{g}$ as follows.
Consider on $\mathfrak{g}$ the Euclidean spheres $S(1,1)$, $S(-1,1)$, $S(-1,-1)$ and $S(1,-1)$ with radius $\sqrt{5}$ and centered on $(1,1)$, $(-1,1)$, $(-1,-1)$ and $(1,-1)$ respectively.
Let $S_e$ be the figure formed by the arcs of:
\begin{itemize}
\item $S_{(-1,-1)}$ intercepted with the first quadrant of $\mathfrak{g}$;
\item $S_{(1,-1)}$ intercepted with the second quadrant of $\mathfrak{g}$;
\item $S_{(1,1)}$ intercepted with the third quadrant of $\mathfrak{g}$ and
\item $S_{(-1,1)}$ intercepted with the fourth quadrant of $\mathfrak{g}$
\end{itemize}
Define $F$ as the left-invariant $C^0$-Finsler structure whose unit sphere on $\mathfrak{g}$ is $S_e$. 
The shape of $S_e=S_F[e,0,1]$ can be seen in Figure \ref{Se}.
We will prove some results before we explain the strong convexity of $F(e,\cdot)$.

\begin{figure}[h]
\centering
\includegraphics[scale=0.7]{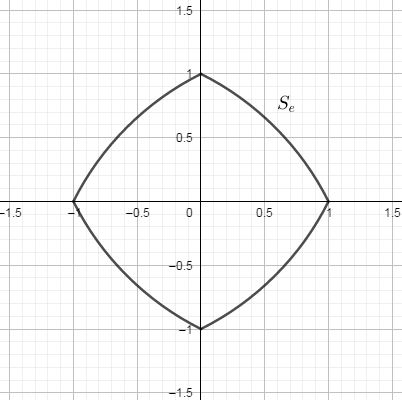}
\caption{The sphere $S_e=S_F[e,0,1]$.}
\label{Se}
\end{figure}

\begin{proposition}
\label{produto interno fortemente convexo em relacao a ele mesmo}
Let $\langle \cdot,\cdot \rangle$ be an inner product on a finite dimensional real vector space $V$. Then the norm $\check F (z):=\sqrt{\langle z,z \rangle}$  is strongly convex with respect to itself.
\end{proposition}
\begin{proof}
The equation
\begin{equation*}
\check F^2(z)=\check F^2(y)+\alpha(z-y)+\check F^2(z-y),
\end{equation*}
is a direct consequence of the properties of a inner product and the fact that $\alpha(\cdot)=2\langle y,\cdot\rangle$.
\end{proof}

\begin{corollary}
A Riemannian metric is strongly convex with respect to itself.
\end{corollary}
\begin{proof}
It is a direct consequence of Proposition \ref{produto interno fortemente convexo em relacao a ele mesmo}.
\end{proof}
\begin{proposition}
\label{fortemente convexo geometrico}
Let $F$ be a norm on a finite dimensional real vector space $V$. 
Suppose that for every $y \in S_F[0,1]$ and $\alpha \in \partial F^2(y)$ there exists a Euclidean inner product $\left< \cdot, \cdot \right>_{y,\alpha}$  on $V$ such that 
\begin{enumerate}
\item $y \in S_{\Vert \cdot \Vert_{y,\alpha}}[0,1]$, where $\Vert \cdot \Vert_{y,\alpha}$ stands for the norm correspondent to $\left< \cdot , \cdot \right>_{y,\alpha}$;
\item $F \geq \Vert \cdot\Vert_{y,\alpha}$;
\item $\ker \partial \alpha (y)$ is parallel to $T_y S_{\Vert \cdot \Vert_{y,\alpha}}[0,1]$;
\item there exists a norm $F_{\min}$ on $V$ such that $F_{\min} \leq \Vert \cdot \Vert_{y,\alpha}$ for every $(y,\alpha)$.
\end{enumerate}
Then $F$ is strongly convex with respect to $F_{\min}$.
\end{proposition}
\begin{proof}
We must prove that
\begin{equation}
\label{convexidade forte criterio}
F^2(z)-F^2(y)\geq \alpha(z-y) + F_{\min}^2(z-y)
\end{equation}
for every $y,z \in V$ and $\alpha \in \partial F^2(y)$.

If $y=0$, then $\alpha = 0$ is the unique point in $\partial F^2(y)$ and (\ref{convexidade forte criterio}) holds trivially.

Let us analyze the case $y \in S_F[0,1]$. 
Consider $\alpha \in \partial F^2 (y)$ and denote $\check F=\Vert \cdot \Vert_{y,\alpha}$. Notice that
\begin{align}
F^2(z)-F^2(y) & \geq \check F^2(z) - \check F^2(y) \geq \left<2y, z-y\right>_{x,\alpha} + \check F^2 ( z -y) \nonumber \\ 
& = d\check F_y(z-y) + \check F^2 ( z -y) \label{inequacao1convexidadeforte}
\end{align}
holds due to Proposition \ref{produto interno fortemente convexo em relacao a ele mesmo} and Items 1 and 2. 

Let us show that $\alpha = d\check F^2_y$.
The functionals $\alpha$ and $d\check F^2_y$ share the same kernel due to Item 3. Notice that $y$ is transversal to $\ker\alpha$ and 
\[
t \mapsto F^2(y+ty) \text{ and } t \mapsto \check F^2(y+ty) 
\]
are both quadratic functions that coincides at $t=\{-2,-1,0\}$, what implies that $F^2(y+ty) = \check F^2(y+ty)$ for every $t\in \mathbb R$. Then
\[
\alpha (y) = \lim_{t\rightarrow 0} \frac{F^2(y+ty) - F^2(y)}{t} = \lim_{t\rightarrow 0} \frac{\check F^2(y+ty) - \check F^2(y)}{t} = (d\check F^2)_y (y),
\]
and $\alpha = d\check F^2_y$ because they coincides in $\ker \alpha$ and $\spann \{y\}$.
Hence
\[
F^2(z)-F^2(y)\geq \alpha(z-y) + \check F^2 ( z -y) \geq \alpha(z-y) + F_{\min}^2(z-y)
\]
due to (\ref{inequacao1convexidadeforte}) and Item 4. 

For an arbitrary $y\in V\backslash 0$, notice that $\alpha \in \partial F^2(y)$ iff $\lambda \alpha \in \partial F^2 ( \lambda y)$ for every $\lambda >0$. 
Therefore if we choose $\lambda = 1/F(y)$, then the former case implies that the inequality
\[
F^2(\tilde z)-F^2(\lambda y) \geq \lambda \alpha(\tilde z- \lambda y) + F_{\min}^2(\tilde z-\lambda y)
\]
holds for every $\tilde z \in V$ and $\alpha \in \partial F^2(y)$, what settles (\ref{convexidade forte criterio}).
\end{proof}

We are going to use Theorem \ref{localmente lipschitz homogeneo} and Proposition \ref{fortemente convexo geometrico} in order to outline the proof of the strong convexity of $F$.
Consider Figure \ref{Direcoes preferenciais}.
The curvature of the smooth part of $S_e$ is $\kappa_{S_e} = \frac{1}{\sqrt{5}}$.
Consider an ellipse $E$ centered at the origin with the semi-minor equals to one and the semi-major $R$ so large such that $\kappa_E(p) \leq \frac{1}{2\sqrt{5}}$ for every $p \in B_{\Vert \cdot \Vert}[0,2] \cap E$. 
We claim that $E$, its homotheties and rotations works as the spheres $S_{\Vert \cdot \Vert_{y,\alpha}}[0,1]$ of Proposition \ref{fortemente convexo geometrico}. 
In fact, choose $y\in S_F[0,1]$ and $\alpha \in \partial F^2(y)$. Let $\tilde E$ an appropriate rotation and homothety of $E$ such that $\tilde E$ is tangent to $S_e$ in $y=\tilde p_{\min}$, where $\tilde p_{\min}$ is a point of a semi-minor of $\tilde E$. It is clear that $S_e$ remains inside $\tilde E$ and consequently the norm $\Vert \cdot\Vert_{y,\alpha}$ with unit sphere $\tilde E$ satisfies all conditions of Proposition \ref{fortemente convexo geometrico}.
Finally the norm $F_{\min}$ such that its unit circle is $S_{\Vert \cdot\Vert}[0,R]$ satisfies the conditions of Proposition \ref{fortemente convexo geometrico}. 
Therefore $\mathcal{E}$ is a locally Lipschitz vector field due to Theorem \ref{localmente lipschitz homogeneo}.

Now we do a geometric and intuitive analysis of how minimizing paths behave on $(G,F)$. The analysis is done in the same spirit of the former example: the point $(x(t),\alpha(t))$ will determine $u(\alpha)$ that gives the direction of the curve. From this relationship we will be able to control the the minimizing path $x(t)$.

Let $x\in G$ and denote the preferred directions in $S_F[x,0,1]$ by $V_x(N)$, $V_x(W)$, $V_x(S)$ and $V_x(E)$ (see Figure \ref{Direcoes preferenciais}).
\begin{figure}[h]
\centering
\includegraphics[scale=0.7]{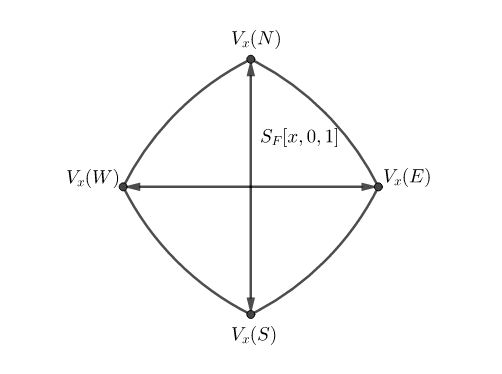}
\caption{Preferred directions.}
\label{Direcoes preferenciais}
\end{figure}

First of all $\alpha_1(t)$ is constant. Assume that $\alpha_1(0)>0$.

For $\alpha_2(0)$ large enough, $V_x(N)$ should maximize $\alpha(0)$. 
In order to see this fact, note that $\alpha(0)=(\alpha_1(0), \alpha_2(0))$ is orthogonal to its level sets with respect to the Euclidean inner product and they are increasing in the direction of $\alpha$. Therefore there exists a level set that is tangent to $S_F[x,0,1]$ exactly at $V_x(N)$ (see Figure \ref{Kernel of xi}).

\begin{figure}[h]
\centering
\includegraphics[scale=0.6]{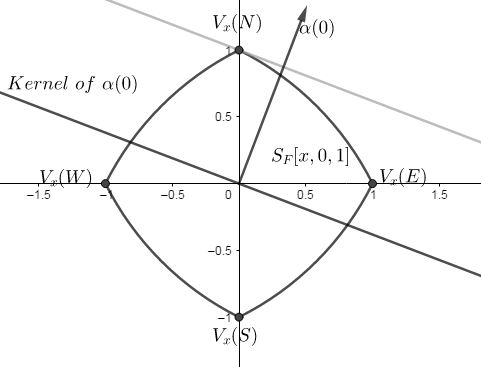}
\caption{$V_x(N)$ maximizes $\alpha(0)$.}
\label{Kernel of xi}
\end{figure}

Now we determine the minimizing path. The calculations below are illustrated in Figure \ref{trajetoria tipica}.

Let $k_1>0$ be such that $V_x(N)$ maximizes $\alpha$ if and only if $\alpha_2\in [k_1,\infty)$. If $\alpha_2(0)\in [k_1,\infty)$, then (\ref{equacao Hamilton}) can be explicitly integrated and we have the solution  $(x^1(t),x^2(t),\alpha_1(t),\alpha_2(t))=(x^1(0),x^2(0)e^t, \alpha_1(0),$ $\alpha_2(0)e^{-t})$ in the maximum interval $\left(-\infty,\ln \left( \frac{\alpha_2(0)}{k_1} \right) \right]$ where $u(\alpha(t))=V_x(N)$.
In this case, $(x^1(t),$ $x^2(t))$ is a piece of vertical line pointed upwards.

Likewise $V_x(S)$ maximizes $\alpha$ if and only if $\alpha_2 \in (-\infty,-k_1]$.
If $\alpha_2(0) \in (-\infty,-k_1]$, then
(\ref{equacao Hamilton}) can be explicitly integrated and we have the solution  $(x^1(t),x^2(t),\alpha_1(t),$ $\alpha_2(t)) = (x^1(0),x^2(0)e^{-t}, \alpha_1(0),$ $\alpha_2(0)e^{t})$ in the maximum interval $\left[\ln \left( \frac{-k_1}{\alpha_2(0)} \right), \infty \right)$ where $u(\alpha(t))=V_x(S)$. 
Here $(x^1(t),x^2(t))$ is a piece of vertical line pointed downwards.

Similarly there exists $k_2>0$ such that $V_x(E)$ maximizes $\alpha(0)$ if and only if $\alpha_2(0)\in [-k_2,k_2]$. If $\alpha_2(0) \in [-k_2,k_2]$, then
(\ref{equacao Hamilton}) can be explicitly integrated and we have the solution  $(x^1(t),x^2(t),\alpha_1(t),$ $\alpha_2(t))=(x^1(0)+x^2(0)t,x^2(0), \alpha_1(0),\alpha_2(0)-\alpha_1(0)t)$ in the maximum interval $\left[\frac{\alpha_2(0)-k_2}{\alpha_1(0)}, \frac{\alpha_2(0)+k_2}{\alpha_1(0)} \right]$ where $u(\alpha(t))=V_x(E)$.
In this case $(x^1(t),x^2(t))$ is a piece of horizontal line pointed to the right-hand side.

If $\alpha_2 \in [k_2,k_1]$, then the unit vector $X_u(x)$ that maximizes $\alpha$ is in the first quadrant. 
As $t$ increases and $\alpha_2(t)$ varies between $k_1$ and $k_2$, $u(t)$ varies between $V_x(N)$ and $V_x(E)$ rotating clockwise. 
Therefore the resulting curve $(x^1(t),x^2(t))$ is similar to a fourth of a circle in the second quadrant with its angular coordinate varying from $\pi$ to $\pi /2$.

Finally if $\alpha_2\in [-k_1,-k_2]$, then the analysis is similar to the case $\alpha_2 \in [k_2,k_1]$, and $(x^1(t),x^2(t))$ is similar to a fourth of a circle in the first quadrant with its angular coordinate varying from $\pi /2$ to $0$.

We can join all parts above and conclude that the minimizing paths of $(G,F$) have the form given in Figure \ref{trajetoria tipica}.

The case $\alpha_1(0)<0$ is treated in the same way, and what we get is a minimizing path as is Figure \ref{trajetoria tipica}, but oriented counterclockwise.

If $\alpha_1(0)\equiv 0$, then for $\alpha_2(0)>0$ we can integrate (\ref{equacao Hamilton}) explicitly and  we have the solution $(x^1(t),x^2(t),$ $\alpha_1(t),\alpha_2(t))=(x^1(0),x^2(0)e^t, 0,\alpha_2(0)e^{-t})$ extendable for every $t\in \mathbb R$.
For $\alpha_2(0)<0$, we also can integrate (\ref{equacao Hamilton}) explicitly and we have the solution $(x^1(t),x^2(t),\alpha_1(t),\alpha_2(t))=(x^1(0),x^2(0)e^{-t}, 0,\alpha_2(0)e^t)$ extendable for every $t \in \mathbb R$.
They are vertical lines.

\begin{figure}[h]
\centering
\includegraphics[scale=0.5]{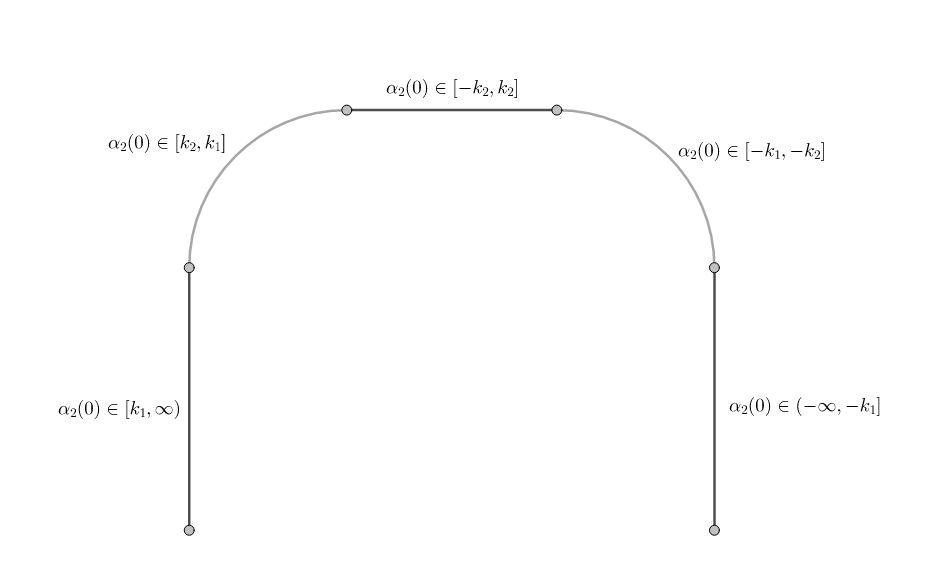}
\caption{A typical minimizing path on $(G,F)$.}
\label{trajetoria tipica}
\end{figure}

\begin{remark}
\label{conclusoes 2}
It is straightforward that if $(x^1(t),x^2(t),\alpha_1(t),\alpha_2(t))$ is a solution of (\ref{equacao Hamilton}), $C_1 >0$ and $C_2 \in \mathbb R$, then $(C_1 x^1(t) + C_2, C_1 x^2(t), \alpha_1(t), \alpha_2(t))$ is also a solution of (\ref{equacao Hamilton}) with the Hamiltonian being the constant $C_1.C_0$ along the curve.
If $(x_0,y_0)$ are such that $y_0$ is a preferred direction, then there exist infinitely many minimizing paths $\gamma:\mathbb R \rightarrow G$ such that $\gamma(0) = x_0$ and $\gamma^\prime(0)=y_0$.
Therefore the minimizing paths of this example can't be represented as the projection of integral curves of a locally Lipschitz vector field on $TM$. 
\end{remark}

\subsection{Maximum of Finsler structures}

\label{maximum-Finsler}

Let $\mathcal F = \{F_1, \ldots, F_k\}$ be a family of Finsler structures on an $n$-dimensional differentiable manifold $M$ and set $F_{\max} = \max\limits_i F_i$.
In this section we show that under independence condition over $\mathcal F$, $(M,F_{\max})$ is a ``local'' Pontryagin type $C^0$-Finsler structure and that $\mathcal E$ is a ``local'' Lipschitz vector field.
At the end of this section we comment how we can ``join'' the local extended geodesic fields in order to define $\mathcal E$ on a maximum subset of $T^\ast M\backslash 0$.

We begin with some prerequisites of convex analysis (see \cite{Urruty-Lemarechal}).

\begin{definition}
\label{indices-ativos}
Let $\mathcal F = \{f_i:V \rightarrow \mathbb R\}_{i=1,\ldots,k}$ be a family of convex functions defined on a finite dimensional real vector space $V$. 
Define $f=\max_{i =1, \ldots, k} f_i$.
The active index set of $\mathcal F$ at $y \in V$ is defined by 
\[
\ind_{\mathcal F}(y) = \{i \in \{ 1, \ldots, k\};f_i(y)=f(y)\}.  
\] 
\end{definition}

\begin{theorem}
\label{maximo-funcoes-convexas}
Let $\mathcal F = \{f_i:V \rightarrow \mathbb R\}_{i=1,\ldots,k}$ and $f$ as in Definition $\ref{indices-ativos}$.
Then
\[
\partial f(y) = \co \left\{\cup_{i \in \ind_{\mathcal F}(y)}\partial f_i(y) \right\},
\]
where $\co$ denotes the convex hull.
\end{theorem}

Now we proceed with our example.

\begin{theorem}
Let $F_1, F_2, \ldots, F_k$ be a family of strongly convex $C^0$-Finsler structures on $M$ with respect to $\check F_1, \ldots, \check F_k$.
Then there exist a $C^0$-Finsler structure $\check F_{\min}$ on $M$ such that $\check F_{\min} \leq \check F_i$ for every $i=1, \ldots, k$.
Moreover $F_{\max} := \max\limits_{i = 1, \ldots, k}F_i$ is a strongly convex $C^0$-Finsler structure on $M$ with respect to $\check F_{\min}$. 
\end{theorem}

\begin{proof}

Let $F$ be any $C^0$-Finsler structure on $M$. 
Define continuous functions $\check c_1, \check c_2: M \rightarrow (0, \infty)$ by
\[
\check c_1(x) = \min\limits_{j\neq i} \min_{y \in S_F[x,0,1]}\frac{\check F_j(x,y)}{\check F_i(x,y)}
\]
and 
\[
\check c_2(x) = \max\limits_{j\neq i} \max_{y \in S_F[x,0,1]}\frac{\check F_j(x,y)}{\check F_i(x,y)}.
\]
Observe that $\check c_1$ and $\check c_2$ does not depend on the choice of $F$.
It follows that
\[
\check c_1(x) \check F_i(x,y) \leq  \check F_j (x,y) \leq \check c_2(x) \check F_i(x,y) 
\]
for every $i,j=1, \ldots, k$ and every $(x,y)\in TM$. 
Finally set $\check F_{\min}(x,y)=\check c_1(x) \check F_1(x,y)$ (or any other $\check c_1(x)\check F_i(x,y))$, what settles the first statement.

In order to prove the last statement, it is straightforward that $F_{\max}$ is a $C^0$-Finsler structure on $M$.
Moreover we have that
\begin{equation}
\label{fortemente convexo com i}
F^2_i(x,z) \geq F^2_i(x, y) + \alpha_i(z-y) + F^2_{\min}(x, z-y)
\end{equation}
for every $x \in M$, $y\in T_xM\backslash 0$, $z \in T_xM$ and $i=1,\ldots, k$ (Here $\alpha_i$ is the differential of the smooth function $y^\prime \mapsto F^2_i(x,y^\prime)$ at $y$). 
If $\alpha \in \partial F_{\max}(x,y)$ and $\ind$ is the active index set of $\{y^\prime \mapsto F_i(x,y^\prime)\}_{i=1, \ldots, k}$ at $y$, then $\alpha$ is a convex combination $\alpha = \sum_{i \in \ind}\lambda_i \alpha_i$ due to Theorem \ref{maximo-funcoes-convexas}.
Considering the correspondent convex combination on (\ref{fortemente convexo com i}), we have that
\[
F^2_{\max}(x,z) \geq \sum_{i\in \ind}\lambda_i F_i^2(x,z) \geq F^2_{\max}(x, y) + \alpha(z-y) + F^2_{\min}(x, z-y),
\]
what settles the theorem.
\end{proof}

The following definition is classical and can be found in \cite{Warner}.

\begin{definition}
\label{define independente} Let $M$ be a differentiable manifold and $x_0 \in M$. 
A set of $C^1$ functions $f_1, \ldots, f_k: M \rightarrow \mathbb R$ is independent on a neighborhood $V$ of $x_0 \in M$ if $\{(df_1)_x, \ldots, (df_k)_x\} \subset T^\ast_xM$ is linearly independent for every $x \in V$. 
\end{definition}

Denote the differential of $y^\prime \mapsto F_i(x,y^\prime)$ at $y$ by $d(F_i(x,\cdot))_y$.

\begin{definition}
\label{finsler transversal}
Let $M$ be an $n$-dimensional differentiable manifold and $\mathcal F = \{F_1, \ldots, F_k\}$ be a finite family of Finsler structures on $M$. Consider $(x_0,y_0) \in TM$ such that $F_i(x_0,y_0) = 1$ for every $i=1,\ldots, k$.
We say that $\mathcal F$ is independent in a neighborhood $W$ of $(x_0,y_0)$ if $\{ d(F_1(x, \cdot))_y, \ldots, d(F_k(x, \cdot))_y\}$ $\subset T_xM$ is linearly independent for every $(x,y) \in W$.
\end{definition}

The independence in Definition \ref{finsler transversal} implies that the unit spheres $S_{F_1}[x_0, 0, 1], \ldots, S_{F_k}[x_0, 0, 1]$ are mutually transversal at $(x_0,y_0)$.
Moreover this kind of transversality is a stable property and it also holds in a neighborhood of $(x_0,y_0)$ with the respective spheres $S_{F_1}[x,0,r_1], \ldots, S_{F_k}[x,0,r_k]$, where the $r_i$'s are sufficiently close to $1$.

\begin{lemma}
\label{maximo-Finsler-Pontryagin}
Let $\mathcal F=\{F_1, \ldots, F_k\}$ be a family of Finsler structures on an $n$-dimensional differentiable manifold $M$ which are independent  in a neighborhood $W$ of $(x_0,y_0)$ and set $F_{\max} = \max\limits_i F_i$.
Suppose that $F_i(x_0,y_0)=1$ for every $i=1,\ldots, k$.
Then there exist a control region $C \subset S^{n-1}$ with non-empty interior and a family of smooth unit vector fields $\{X_u, u \in C\}$ on a neighborhood $U$ of $x_0 \in (M, F_{\max})$ such that $u \mapsto X_u(x)$ is a homeomorphism of $C$ onto its image in $S_{F_{\max}}[x,0,1]$. 
Moreover $\{X_u\}$ satisfies conditions (2) and (3) of Definition \ref{horizontally smooth}.
\end{lemma}

\begin{proof}
Let $\phi=(x^1, \ldots, x^n)$ be a coordinate system on a neighborhood $\tilde U$ of $x_0\in M$ and $\phi_{T\tilde U} = (x^1, \ldots, x^n, y^1,$  $\ldots, y^n)$ be the corresponding coordinate system on $T\tilde U$.
Let $F_{k+1}, \ldots, F_n: W \rightarrow \mathbb R$ be smooth functions on $W$ such that $\{ d(F_1(x_0, \cdot))_{y_0}, \ldots, d(F_n(x_0, \cdot))_{y_0}\}$ is a basis of $T_{x_0}M$.
Therefore there exist a neighborhood $W_1$ of $(x_0,y_0) \in TM$ such that 
\begin{equation}
\label{define-psi}
\psi(x,y) = ( x^1, \ldots, x^n, F_1(x,y), \ldots, F_n(x, y) ):=(x^1, \ldots, x^n, u^1, \ldots, u^n)
\end{equation}
is a coordinate system on $W_1$. We can suppose, without loss of generality, that  $F_i(x_0, y_0) = 1/2$ for every $i = k+1, \ldots n$.

Instead of considering $S^{n-1}$ as the Euclidean sphere, we set 
\[
S^{n-1} = \{(z^1,\ldots, z^n) \in \mathbb R^{n}; \max\limits_i \vert z^i\vert = 1\},
\]
which is the unit sphere centered at the origin with respect to the maximum norm.

Consider a neighborhood $U$ of $x_0 \in M$ and $\varepsilon \in (0,1/2)$ such that 
\[
W_2 = \phi(U) \times Q := \phi (U) \times (F_1(x_0,y_0) -\varepsilon, F_1(x_0,y_0) + \varepsilon) \times \ldots \times (F_n(x_0,y_0) -\varepsilon, F_n(x_0,y_0) + \varepsilon)  \subset \subset \psi(W_1).
\]
Observe that
\begin{equation}
\label{xu-u-suave}
X_u(x):=(\psi^{-1}\vert_{W_2})(x,u)
\end{equation}
is a smooth family of vector fields on $U$ (not necessarily unit vector fields) that satisfies conditions (2) and (3) of Definition \ref{horizontally smooth}.
Finally if we set $C= S^{n-1} \cap Q$,
then $C$ is a neighborhood of $(F_1(x_0,y_0), \ldots, F_n(x_0,y_0)) \in S^{n-1}$ and 
\[
\{X_u(x) = \psi^{-1}(x,u)\}_{u\in C}
\]
is a family of unit vector fields on $U$ satisfying Items (2) and (3) of Definition \ref{horizontally smooth}.

\end{proof}

Suppose that we are in the conditions of Lemma \ref{maximo-Finsler-Pontryagin}.
We also suppose that all elements of the proof of Lemma \ref{maximo-Finsler-Pontryagin} are in place. 
The strong convexity of $F_{\max}$ implies that $\mathcal E$ is a vector field on $T^\ast M$. 
Suppose that $(x_0,\alpha_0) \in T^\ast M$ is such that $y_0 = X_{u(x_0,\alpha_0)}(x_0)$ is its unique maximizer  in $S_{F_{\max}}[x_0, 0 ,1]$.
We are going to prove a Lipschitz type inequality
\[
\Vert \mathcal E(x,\alpha) - \mathcal E(x_0, \alpha_0) \Vert \leq C_0\left( \Vert x-x_0\Vert + \Vert \alpha - \alpha_0 \Vert \right)
\]
in order to prove that $\mathcal E$ is a locally Lipschitz vector field, where $\Vert \cdot \Vert$ is the canonical Euclidean norm with respect to the coordinate system $\phi_{T^\ast U}$. 
In what follows we determine a neighborhood $U_0 \times V^\ast_0$ of $(x_0, \alpha_0)$ in $T^\ast M$ such that $u(x, \alpha)$ is well defined for every $(x,\alpha) \in U_0 \times V^\ast_0$.
Moreover we define a smooth $1$-form $\tilde \alpha_0(x)$ in a neighborhood of $x_0$, which is a extension of $\alpha_0$, such that the triangle inequality
\begin{equation}
\label{desigualdade-triangular-maximo}
\Vert \mathcal E(x,\alpha) - \mathcal E(x_0, \alpha_0)\Vert 
\leq \Vert \mathcal E(x,\alpha) - \mathcal E(x, \tilde \alpha_0(x))\Vert + \Vert \mathcal E(x,\tilde \alpha_0 (x)) - \mathcal E(x_0, \alpha_0)\Vert
\end{equation}
is useful. 
The second term of the right-hand side of (\ref{desigualdade-triangular-maximo}) is a kind of horizontal control of the left-hand side and the first term of the right-hand side of (\ref{desigualdade-triangular-maximo}) is a kind of vertical control.

Let us begin with the definition of $\tilde \alpha_0$ and the horizontal control of (\ref{desigualdade-triangular-maximo}).

A functional $\alpha \in T^\ast_x M$ is maximized by $y$ in $S_{F_{\max}}[x,0,1]$ iff $\alpha$ is a positive multiple of an element of $\partial F_{\max}(x,y)$ due to the strong convexity of $F_{\max}$ and its positive homogeneity.
In this case, we can write 
\begin{equation}
\label{notacao-soma-alpha}
\alpha = \sum_{i \in \ind } c^\prime_i d(F_i(x,\cdot))_{y}, \;\; c^\prime_i \geq 0, \;\; i\in \ind,
\end{equation} 
where $\ind$ is the active index set of $\{F_1(x,\cdot), \ldots, F_k(x,\cdot)\}$ at $y$, due to  Theorem \ref{maximo-funcoes-convexas}.

Consider $\alpha_0 \in T^\ast_{x_0} M$ that is maximized by $y_0 \in S_{F_{\max}}[x_0,0,1]$.
We can write
\[
\alpha_0 = \sum_{i=1}^k c_i d(F(x_0,\cdot))_{y_0} = \sum_{i=1}^k c_i d(F(x_0,\cdot))_{X_{u(x_0, \alpha_0)}(x_0)}
\]
as in (\ref{notacao-soma-alpha}) because $\ind = \{1, 2, \ldots, k\}$.
Define the smooth extension
\[
\tilde \alpha_0 (x) = \sum_{i=1}^k c_id(F_i(x,\cdot))_{X_{u(x_0,\alpha_0)}(x)}
\]
of $\alpha_0$ on $U$.
It follows that there exist $C_1>0$ such that
\begin{equation}
\label{exemplo-controle-horizontal}
\Vert \tilde \alpha_0 (x) - \alpha_0 \Vert \leq C_1 \Vert x-x_0\Vert
\end{equation}
due to the relative compactness of $U$.
Moreover
\[
\mathcal E(x,\tilde \alpha_0(x)) = \vec H_{u(x_0, \alpha_0)}(x, \tilde \alpha_0 (x))
\] 
and we have that
\begin{align}
\Vert \mathcal E(x,\tilde \alpha_0(x)) - \mathcal E(x_0, \alpha_0)\Vert & \leq \Vert \vec H_{u(x_0, \alpha_0)}(x, \tilde \alpha_0 (x)) - \vec H_{u(x_0, \alpha_0)}(x_0, \alpha_0)\Vert \nonumber \\
& \leq \tilde C\left( \Vert x-x_0\Vert + \Vert \tilde \alpha_0 (x) - \alpha_0\Vert \right) \nonumber \\
& \leq C_2 \Vert x-x_0\Vert, \label{exemplo-estimativa-E-transporte}
\end{align}
due to (\ref{exemplo-controle-horizontal}) and the smoothness of $\vec H_{u(x_0,\alpha_0)}$. 
This is the horizontal estimate of (\ref{desigualdade-triangular-maximo}).

For $u \in Q$, set
\[
\ind_u = \{i \in \{1, \ldots, k\}, u^i \geq u^j \text{ for every }j\in \{1, \ldots, k\}\}
\]
and
\[
u_{\max} = \max\limits_{i=1, \ldots, k} u_i.
\]
Notice that $\ind_u$ is the active index set of $\{F_1(x, \cdot), \ldots, F_k(x, \cdot)\}$ at $X_u(x)$ for every $x \in U$.
Define
\[
\tilde \alpha_{u,a_i} (x) = \sum_{i \in \ind_u} a_i d(F_i(x,\cdot))_{X_u(x)}.
\]
Notice that $X_u(x) \in S_{F_{\max}}[x,0,u_{\max}]$.
Moreover $\tilde \alpha_{u,a_i}$ is maximized by $X_u(x)$ in $S_{F_{\max}}[x,0,u_{\max}]$
 iff $a_i\geq 0$ for every $i \in \ind_u$ due to (\ref{define-psi}) and (\ref{notacao-soma-alpha}).
 
In order to extend (\ref{exemplo-controle-horizontal}) to $\{\tilde \alpha_{u,a_i}\}_{\{u \in Q;a_i\geq 0\}}$, observe that if
\[
\alpha = \sum_{i=1}^k a_i d(F_i(x_0, \cdot))_{y_0},
\] 
then there exist $\tilde C > 0$ such that $a_i \leq \tilde C\Vert \alpha\Vert$ for every $\alpha \in \spann \{d(F_i(x_0, \cdot))_{y_0}\}_{i=1, \ldots, k}$. 
Analogously if
\[
\tilde \alpha_{u,a_i}(x)=\sum_{i=1}^k a_i d(F_i(x, \cdot))_{X_{u}(x)},
\]
then there exists a constant $C_3>0$ satisfying
\begin{equation}
\label{estima-ais}
a_i \leq C_3\Vert \tilde \alpha_{u,a_i}(x)\Vert 
\end{equation}
for every $(x,u) \in U \times Q$ and $(a_1, \ldots, a_k) \in \mathbb R^k$ due to the relative compactness of $U \times Q$.

Now we will construct a relatively compact neighborhood $U_0 \times V^\ast_0$ of $(x_0, \alpha_0) \in T^\ast M$ such that $u(x, \alpha)$ is well defined for every $(x, \alpha) \in U_0 \times V^\ast_0$. Here $U_0 \times V^\ast_0$ is the product with respect to the coordinate system $\phi_{T^\ast \tilde U}$. 
As a consequence $\mathcal E(x,\alpha)$ will be well defined for every $(x,\alpha) \in U_0 \times V^\ast_0$. 
In order to do that, we consider $V^\ast_x = \{\alpha \in T^\ast_x M; u(x,\alpha) \in C\}$ for each $x \in U$. 
Then $V^\ast_x$ is an open subset of $T^\ast_xM$ because $\alpha \mapsto X_{u(x, \alpha)}(x)$ and $X_{u(x,\alpha)} \mapsto u(x,\alpha)$ (with $x$ fixed) are Lipschitz maps (see Theorem \ref{diferencialduallipschitz} and (\ref{define-psi}) respectively).
If we prove that $\cup_{x \in U} V^\ast_x$ is an open subset of $T^\ast M$, then the existence of a neighborhood $U_0 \times V^\ast_0$ of $(x_0, \alpha_0)$ is immediate. 
In order to do that, we prove that the variation of the family $V^\ast_x$ with respect to $x$ is sufficiently well behaved.

Define the family of isomorphisms $\iota (x,u): T^\ast_{x_0} M \rightarrow T^\ast_{x} M$ parametrized by $U \times Q$ as
\[
d(F_i(x_0, \cdot))_{X_{u}(x_0)} \mapsto d(F_i(x, \cdot))_{X_u(x)}, \hspace{3mm} i=1, \ldots, n
\]
on the basis $\{d(F_i(x_0, \cdot))_{X_{u}(x_0)}\}_{i=1, \ldots, n}$ of $T^\ast_{x_0}M$.
Then there exist constants $C_4,C_5 > 0$ such that 
\begin{equation}
\label{estimativa-iota}
C_4 \Vert \alpha \Vert 
\leq \Vert \iota (x, u)(\alpha) \Vert 
\leq C_5.\Vert \alpha \Vert
\end{equation}
for every $(x,u) \in U \times Q$ and every $\alpha \in T^\ast_{x_0}M$ due to the smoothness of the map $(x,u) \mapsto d(F_i(x,\cdot))_{X_u(x)}$ and the relative compactness of $U \times Q$.

Now we define a family of bijections $\zeta_x: V^\ast_{x_0} \rightarrow V^\ast_x$ parametrized by $U$.
Represent $\alpha \in V^\ast_{x_0}$ as  
\[
\alpha = \sum_{i \in \ind_{u(x_0,\alpha)}}c_i d(F_i(x_0, \cdot))_{X_{u(x_0, \alpha)}(x_0)}, \hspace{2mm} c_i \geq 0
\]
and define
\[
\zeta_x(\alpha) = \sum_{i \in \ind_{u(x_0,\alpha)}}c_i d(F_i(x, \cdot))_{X_{u(x_0, \alpha)}(x)}.
\]
Consider
\[
\beta = \sum_{i \in \ind_{u(x_0,\beta)}} \tilde c_i d(F_i(x_0, \cdot))_{X_{u(x_0, \beta)}(x_0)}, \hspace{2mm} \tilde c_i \geq 0.
\]
We have that
\begin{align}
\Vert \zeta_x(\alpha - \beta) \Vert 
& \leq \left\Vert \sum_{i \in \ind_{u(x_0,\alpha)}}c_i d(F_i(x, \cdot))_{X_{u(x_0, \alpha)}(x)} -\sum_{i \in \ind_{u(x_0,\beta)}} \tilde c_i d(F_i(x, \cdot))_{X_{u(x_0, \beta)}(x)} \right\Vert \nonumber \\
& \leq \left\Vert \sum_{i = 1}^k (c_i - \tilde c_i) d(F_i(x, \cdot))_{X_{u(x_0, \alpha)}(x)} \right\Vert \nonumber \\
& + \left\Vert \sum_{i=1}^{k} \tilde c_i \left( d(F_i(x, \cdot))_{X_{u(x_0, \alpha)}(x)} - d(F_i(x, \cdot))_{X_{u(x_0, \beta)}(x)} \right) \right\Vert \nonumber \\
& \leq C_5.\left\Vert \sum_{i = 1}^k (c_i - \tilde c_i) d(F_i(x_0, \cdot))_{X_{u(x_0, \alpha)}(x_0)} \right\Vert \nonumber \\
& + C_3 \Vert \beta \Vert \left\Vert \sum_{i=1}^{k}  \left( d(F_i(x, \cdot))_{X_{u(x_0, \alpha)}(x)} - d(F_i(x, \cdot))_{X_{u(x_0, \beta)}(x)} \right) \right\Vert \nonumber \\
& \leq C_5. \left\Vert \sum_{i=1}^k c_i d(F_i(x_0, \cdot))_{X_{u(x_0, \alpha)}(x_0)} -\sum_{i=1}^k \tilde c_i d(F_i(x_0, \cdot))_{X_{u(x_0, \beta)}(x_0)} \right\Vert \nonumber \\
& + C_5 . \left\Vert \sum_{i=1}^k \tilde c_i \left ( d(F_i(x_0, \cdot))_{X_{u(x_0, \beta)}(x_0)} - d(F_i(x_0, \cdot))_{X_{u(x_0, \alpha)}(x_0)} \right) \right\Vert \nonumber \\
& + C_3 \Vert \beta \Vert \left\Vert \sum_{i=1}^{k}  \left( d(F_i(x, \cdot))_{X_{u(x_0, \alpha)}(x)} - d(F_i(x, \cdot))_{X_{u(x_0, \beta)}(x)} \right) \right\Vert \nonumber \\
& \leq C_5. \Vert \alpha - \beta \Vert + C_5. C_3.\Vert \beta \Vert . \tilde C . \Vert X_{u(x_0,\alpha)}(x_0) - X_{u(x_0,\beta)}(x_0) \Vert \nonumber \\
& + C_3 \Vert \beta \Vert . \tilde C. \Vert X_{u(x_0,\alpha)}(x) - X_{u(x_0,\beta)}(x) \Vert \nonumber \\
& \leq C_6 (1+\Vert \beta \Vert)\Vert \alpha - \beta\Vert \nonumber
\end{align}
due to (\ref{estima-ais}), (\ref{estimativa-iota}), Remark \ref{differencial} and the smoothness of $F_i$.
Therefore $\zeta_x$ are locally Lipschitz bijections such that $V^\ast_x$ vary continuously with respect to $x$, what implies that
$\cup_{x \in U} V_x$ is an open subset of $T^\ast M$ and settles the existence of a relatively compact neighborhood $U_0 \times V^\ast_0$ of $(x_0, \alpha_0) \in T^\ast M$ such that $u(x,\alpha)$ is defined for every $(x,\alpha) \in U_0 \times V^\ast_0$.

Let us estimate the vertical part of (\ref{desigualdade-triangular-maximo}). Consider $(x, \alpha) \in U_0 \times V_0$ and decompose
\[
\mathcal E(x,\alpha) = (X_{u(x,\alpha)}(x), V(x, \alpha))
\]
where 
\[
V(x,\alpha) = - \alpha_j \frac{\partial f^j}{\partial x^i}(x,u(x,\alpha)) \frac{\partial}{\partial \alpha_i}.
\]
We have that
\begin{align}
\Vert \mathcal E(x,\alpha) - \mathcal E(x, \beta) \Vert 
& \leq 
\Vert X_{u(x,\alpha)}(x) - X_{u(x,\beta)}(x)\Vert + \Vert V(x, \alpha)-V(x, \beta) \Vert \nonumber \\ 
& \leq \tilde C_1\Vert \alpha - \beta \Vert \nonumber \\
& + \left\Vert - \alpha_j \frac{\partial f^j}{\partial x^i}(x,u(x,\alpha)) \frac{\partial}{\partial \alpha_i} + \alpha_j \frac{\partial f^j}{\partial x^i}(x,u(x,\beta)) \frac{\partial}{\partial \alpha_i} \right\Vert \nonumber \\ 
& 
+ \left\Vert \left( \beta_j - \alpha_j \right) \frac{\partial f^j}{\partial x^i}(x,u(x,\beta)) \frac{\partial}{\partial \alpha_i} 
\right\Vert \nonumber \\ 
& \leq \tilde C_1 \Vert \alpha - \beta \Vert + \tilde C_2 \Vert \alpha\Vert. \tilde C_3. \Vert u(x,\alpha) - u(x, \beta) \Vert + \tilde C_4\Vert \alpha - \beta \Vert \nonumber \\
& \leq \tilde C_5 \Vert \alpha - \beta \Vert + \tilde C_6 \Vert X_{u(x,\alpha)}(x) - X_{u(x, \beta)}(x)\Vert \nonumber \\
& \leq C_7 \Vert \alpha - \beta \Vert \label{exemplo-estimativa-vertical}
\end{align}
due to Remark \ref{differencial}, (\ref{xu-u-suave}) and to the smoothness of $\partial f^j/\partial x^i$.

Finally consider (\ref{exemplo-controle-horizontal}), (\ref{exemplo-estimativa-E-transporte}), (\ref{exemplo-estimativa-vertical}) and $(x,\alpha)\in U_0 \times V^\ast_0$ in order to get 
\begin{align}
\Vert \mathcal E(x,\alpha) - \mathcal E(x_0, \alpha_0) \Vert & \leq \Vert \mathcal E(x,\alpha) - \mathcal E(x, \tilde \alpha_0(x)) \Vert + \Vert \mathcal E(x,\tilde \alpha_0 (x)) - \mathcal E(x_0, \alpha_0) \Vert \nonumber \\ 
& \leq C_7 \Vert \alpha - \tilde \alpha_0(x)  \Vert  + C_2 \Vert x-x_0\Vert \nonumber \\
& \leq C_7 \left( \Vert \alpha - \alpha_0  \Vert + \Vert \tilde \alpha_0(x) - \alpha_0  \Vert \right) + C_2 \Vert x-x_0\Vert \nonumber \\
& \leq C_0 \left( \Vert \alpha - \alpha_0\Vert + \Vert x- x_0\Vert \right). \nonumber
\end{align}
We have proved the following theorem

\begin{theorem}
\label{teorema-maximo-localmente-Lipschitz}
Suppose that $(M, F_{\max})$ and $(x_0, y_0)$ satisfies the conditions of Lemma \ref{maximo-Finsler-Pontryagin}.
Let $(x_0, \alpha_0) \in T^\ast M$ such that $X_{u(x_0,\alpha_0)}(x_0) = y_0$.
Then $\mathcal E$ is a Lipschitz vector field in a neighborhood of $(x_0, \alpha_0)$.
\end{theorem}

\begin{remark}
\label{exemplo-globalizar}
This example depends only on the theory developed until Section \ref{asymmetric norms}, which holds for this local setting.

The local extended geodesic field $\mathcal E$ does not depend either on the coordinate system or on the control system that is in place.
Therefore we can consider all the open subsets of $T^\ast M\backslash 0$ where $\mathcal E$ can be defined locally and we can join them all in order to define $\mathcal E$ in the maximum subset.

The authors do not know whether Theorem \ref{teorema-maximo-localmente-Lipschitz} holds if $\mathcal F$ is not independent at $(x_0, y_0) \in TM$.
\end{remark}

\section{Suggestions for future works}
\label{conclusoes}

In this chapter we leave some open questions and suggestions for future works:

\begin{enumerate}

\item It would be interesting to find (if possible) a Pontryagin type $C^0$-Finsler structure such that $\mathcal E$ is a vector field that satisfies the condition of local existence of solutions, but it doesn't satisfy the condition of uniqueness.

\item Pontryagin's maximum principle is a useful tool in order to find necessary condition for the problem of minimizing paths and geodesics.
It would be nice if we can find useful sufficient conditions, as it happens locally in Riemannian and Finsler geometry.
Results concerning the injectivity radius of $(M,F)$ would be welcome
(This problem arose from a question made by Prof. Josiney Alves de Souza).    

\item In the strongly convex case, it is natural to study more general hypotheses in order to assure that $\mathcal E$ is a locally Lipschitz vector field;

\item It would be interesting to study integral curves of $\mathcal E$ for particular instances of $G$-invariant $C^0$-Finsler structures on homogeneous spaces.

\item At least for some particular cases, we can try to study curvatures on Pontryagin type $C^0$-Finsler manifolds considering ideas similar to the theory of Jacobi fields.

\item It would be interesting to study traditional geometrical objects such as connections and curvatures on the non-smooth part of the maximum of independent Finsler structures.

\end{enumerate}


\begin{thebibliography}{99}

\bibitem{Agrachev-Barilari-Rizzi}
A.~Agrachev, D.~Barilari, and L.~Rizzi, \emph{Curvature: a variational
  approach}, Mem. Amer. Math. Soc. \textbf{256} (2018), no.~1225, v+142.

\bibitem{Agrachev-Gamkrelidze-feedback-1}
A.~A. Agrachev and R.~V. Gamkrelidze, \emph{Feedback-invariant optimal control
  theory and differential geometry. {I}. {R}egular extremals}, J. Dynam.
  Control Systems \textbf{3} (1997), no.~3, 343--389.

\bibitem{Agrachev-Barilari-Boscain}
Andrei Agrachev, Davide Barilari, and Ugo Boscain, \emph{A comprehensive
  introduction to sub-{R}iemannian geometry}, Cambridge Studies in Advanced
  Mathematics, vol. 181, Cambridge University Press, Cambridge, 2020, From the
  Hamiltonian viewpoint, With an appendix by Igor Zelenko.

\bibitem{Agrachev-Lee}
Andrei Agrachev and Paul Lee, \emph{Optimal transportation under nonholonomic
  constraints}, Trans. Amer. Math. Soc. \textbf{361} (2009), no.~11,
  6019--6047.

\bibitem{Agrachev-Barilari-Paoli}
Andrei~A. Agrachev, Davide Barilari, and Elisa Paoli, \emph{Volume geodesic
  distortion and {R}icci curvature for {H}amiltonian dynamics}, Ann. Inst.
  Fourier (Grenoble) \textbf{69} (2019), no.~3, 1187--1228.

\bibitem{Sachkov}
Andrei~A. Agrachev and Yuri~L. Sachkov, \emph{Control theory from the geometric
  viewpoint}, Encyclopaedia of Mathematical Sciences, vol.~87, Springer-Verlag,
  Berlin, 2004, Control Theory and Optimization, II.

\bibitem{Ardentov-LeDonne-Sachkov}
A.~A. Ardentov, \`E Le Donne  and Yu.~L. Sachkov, \emph{A sub-{F}insler problem on the {C}artan group},  Tr. Mat. Inst. Steklova \textbf{304} (2019), Optimal\cprime noe Upravlenie i Differentsial\cprime nye Uravneniya, 49–67.

\bibitem{Ardentov-Lokutsievskiy-Sachkov}
A.~A. Ardentov, L.~V. Lokutsievskiy, and Yu.~L. Sachkov, \emph{Extremals for a
  series of sub-{F}insler problems with 2-dimensional control via convex
  trigonometry}, ESAIM Control Optim. Calc. Var. \textbf{27} (2021), Paper No.
  32, 52.

\bibitem{BaoChernShen}
D.~Bao, S.-S. Chern, and Z.~Shen, \emph{An introduction to {R}iemann-{F}insler
  geometry}, Graduate Texts in Mathematics, vol. 200, Springer-Verlag, New
  York, 2000.

\bibitem{Barilari-Chitour}
D.~Barilari, Y.~Chitour, F.~Jean, D.~Prandi, and M.~Sigalotti, \emph{On the
  regularity of abnormal minimizers for rank 2 sub-{R}iemannian structures}, J.
  Math. Pures Appl. (9) \textbf{133} (2020), 118--138.

\bibitem{Barilari-Rizzi-conjugate}
D.~Barilari and L.~Rizzi, \emph{Comparison theorems for conjugate points in
  sub-{R}iemannian geometry}, ESAIM Control Optim. Calc. Var. \textbf{22}
  (2016), no.~2, 439--472.

\bibitem{Barilari-Boscain-LeDonne-Sigalotti}
Davide Barilari, Ugo Boscain, Enrico Le~Donne, and Mario Sigalotti,
  \emph{Sub-{F}insler structures from the time-optimal control viewpoint for
  some nilpotent distributions}, J. Dyn. Control Syst. \textbf{23} (2017),
  no.~3, 547--575.

\bibitem{Barilari-Rizzi-Inventiones}
Davide Barilari and Luca Rizzi, \emph{Sub-{R}iemannian interpolation
  inequalities}, Invent. Math. \textbf{215} (2019), no.~3, 977--1038.

\bibitem{Berestovskii1}
V.~N. Berestovski\u\i, \emph{Homogeneous manifolds with an intrinsic metric.
  {I}}, Sibirsk. Mat. Zh. \textbf{29} (1988), no.~6, 17--29.

\bibitem{Berestovskii2}
\bysame, \emph{Homogeneous manifolds with an intrinsic metric. {II}}, Sibirsk.
  Mat. Zh. \textbf{30} (1989), no.~2, 14--28, 225.

\bibitem{Berestovskii-Zubareva-Engel}
V.~N. Berestovski\u{\i} and I.~A. Zubareva, \emph{Extremals of a left-invariant
  sub-{F}insler metric on the {E}ngel group}, Sibirsk. Mat. Zh. \textbf{61}
  (2020), no.~4, 735--751. 

\bibitem{Boscain-Rossi}
Ugo Boscain and Francesco Rossi, \emph{Invariant {C}arnot-{C}aratheodory
  metrics on {$S^3,\ {\rm SO}(3),\ {\rm SL}(2)$}, and lens spaces}, SIAM J.
  Control Optim. \textbf{47} (2008), no.~4, 1851--1878. 

\bibitem{Burago}
Dmitri Burago, Yuri Burago, and Sergei Ivanov, \emph{A course in metric
  geometry}, Graduate Studies in Mathematics, vol.~33, American Mathematical
  Society, Providence, RI, 2001.

\bibitem{Cobzas}
\c{S}tefan Cobza\c{s}, \emph{Functional analysis in asymmetric normed spaces},
  Frontiers in Mathematics, Birkh\"{a}user/Springer Basel AG, Basel, 2013.

\bibitem{CordovaFukuokaNeves}
N.~Cordova, R.~Fukuoka, and Neves~E. A., \emph{Sequence of induced {H}ausdorff
  metrics on {L}ie groups}, Bull. Braz. Math. Soc. (N.S.) \textbf{51} (2020),
  509--530.

\bibitem{Fukuoka-large-family}
Ryuichi Fukuoka, \emph{A large family of projectively equivalent
  ${C}^0$-{F}insler manifolds}, Tohoku Math. J. \textbf{72} (2020), no.~3,
  725--750.

\bibitem{Fukuoka-Setti}
Ryuichi Fukuoka and Anderson~Macedo Setti, \emph{Mollifier smoothing of
  {$C^0$}-{F}insler structures}, Ann. Mat. Pura Appl. (4) \textbf{200} (2021),
  no.~2, 595--639.

\bibitem{Gribanova}
I.~A. Gribanova, \emph{The quasihyperbolic plane}, Sibirsk. Mat. Zh.
  \textbf{40} (1999), no.~2, 288--301, ii.

\bibitem{Hakavuori-Infinite-geodesics}
Eero Hakavuori, \emph{Infinite geodesics and isometric embeddings in {C}arnot
  groups of step 2}, SIAM J. Control Optim. \textbf{58} (2020), no.~1,
  447--461.

\bibitem{Urruty-Lemarechal}
Jean-Baptiste Hiriart-Urruty and Claude Lemar\'{e}chal, \emph{Fundamentals of
  convex analysis}, Grundlehren Text Editions, Springer-Verlag, Berlin, 2001,
  Abridged version of {{\i}t Convex analysis and minimization algorithms. I}
  [Springer, Berlin, 1993; MR1261420 (95m:90001)] and {{\i}t II} [ibid.;
  MR1295240 (95m:90002)].

\bibitem{Lee-displacement}
Paul W.~Y. Lee, \emph{Displacement interpolations from a {H}amiltonian point of
  view}, J. Funct. Anal. \textbf{265} (2013), no.~12, 3163--3203.

\bibitem{Lokutsievskiy-convex}
L.~V. Lokutsievskiy, \emph{Convex trigonometry with applications to
  sub-{F}insler geometry}, arXiv:1807.08155 (2020).

\bibitem{MatveevTroyanov}
Vladimir~S. Matveev and Marc Troyanov, \emph{The {B}inet-{L}egendre metric in
  {F}insler geometry}, Geom. Topol. \textbf{16} (2012), no.~4, 2135--2170.

\bibitem{Mennucci-asymmetric-distances}
Andrea C.~G. Mennucci, \emph{On asymmetric distances}, Anal. Geom. Metr. Spaces
  \textbf{1} (2013), 200--231.

\bibitem{Mennucci-geodesics}
\bysame, \emph{Geodesics in asymmetric metric spaces}, Anal. Geom. Metr. Spaces
  \textbf{2} (2014), no.~1, 115--153.

\bibitem{Ohta-Hamiltonian}
Shin-ichi Ohta, \emph{On the curvature and heat flow on {H}amiltonian systems},
  Anal. Geom. Metr. Spaces \textbf{2} (2014), no.~1, 81--114.

\bibitem{Pontryagin}
L.~S. Pontryagin, V.~G. Boltyanskii, R.~V. Gamkrelidze, and E.~F. Mishchenko,
  \emph{The mathematical theory of optimal processes}, Translated by D. E.
  Brown, A Pergamon Press Book. The Macmillan Co., New York, 1964.

\bibitem{RockafellarTyrrell}
R.~Tyrrell Rockafellar, \emph{Convex analysis}, Princeton Mathematical Series,
  No. 28, Princeton University Press, Princeton, N.J., 1970.

\bibitem{Royden}
H.~L. Royden, \emph{Real analysis}, third ed., Macmillan Publishing Company,
  New York, 1988.

\bibitem{Sachkov-bang-bang-Cartan}
Yu. Sachkov, \emph{Optimal bang-bang trajectories in sub-{F}insler problem on
  the {C}artan group}, Russ. J. Nonlinear Dyn. \textbf{14} (2018), no.~4,
  583--593.

\bibitem{Sachkov-bang-bang-Engel}
\bysame, \emph{Optimal bang-bang trajectories in sub-{F}insler problems on the
  {E}ngel group}, Russ. J. Nonlinear Dyn. \textbf{16} (2020), no.~2, 355--367.

\bibitem{Setti}
Anderson~Macedo Setti, \emph{Smoothing of ${C}^0$-{F}insler structures}, Ph.D.
  thesis, State University of Maring\'a, 2019, State University of Maring\'a,
  In Portuguese.

\bibitem{Warner}
Frank~W. Warner, \emph{Foundations of differentiable manifolds and {L}ie
  groups}, Graduate Texts in Mathematics, vol.~94, Springer-Verlag, New
  York-Berlin, 1983, Corrected reprint of the 1971 edition.

\bibitem{Zelenko-Li}
Igor Zelenko and Chengbo Li, \emph{Differential geometry of curves in
  {L}agrange {G}rassmannians with given {Y}oung diagram}, Differential Geom.
  Appl. \textbf{27} (2009), no.~6, 723--742.

\end{thebibliography}
\end{document}